\documentclass[10pt]{amsart}
\calclayout

\usepackage{enumerate}
\usepackage{color}
\usepackage{amsthm,amssymb}
\usepackage{pdfpages}
\usepackage{graphicx}
\usepackage{hyperref}
\usepackage[authoryear]{natbib}
\newtheorem{theorem}{Theorem}

\newtheorem{lemma}[theorem]{Lemma}
\newtheorem{corollary}[theorem]{Corollary}

\newtheorem{assumption}[theorem]{Assumption}

\theoremstyle{remark}
\newtheorem{remark}[theorem]{Remark}
\newtheorem*{remark*}{Remark}
\newtheorem{example}[theorem]{Example}

\usepackage{color}

\newcommand{\R}{\mathbb{R}}

\newcommand{\N}{\mathbb{N}}

\newcommand{\E}{\mathbb{E}}

\renewcommand{\P}{\mathbb{P}} %% my computer doesn't know mathds

\newcommand{\coloneq}{\mathrel{\mathop:}=}

\renewcommand{\epsilon}{\varepsilon}

\newcommand{\bP}{\mathbb{P}}
\newcommand{\bF}{\mathbb{F}}

\newcommand{\cA}{\mathcal{A}}

\newcommand{\cP}{\mathcal{P}}
\newcommand{\cS}{\mathcal{S}}

\usepackage{xr}
\newcommand{\eqreff}[1]{[\ref{#1}]}

\usepackage{tikz}

\newcommand{\opt}{{\star}}
\newcommand{\Aoptim}[1]{\cA^\opt_{#1}}
\newcommand{\Boptim}{B^\opt}
\newcommand{\interior}[1]{{#1}^{o}}

\DeclareFontFamily{U}{rcjhbltx}{}
\DeclareFontShape{U}{rcjhbltx}{m}{n}{<->rcjhbltx}{}
\DeclareSymbolFont{hebrewletters}{U}{rcjhbltx}{m}{n}
\DeclareMathSymbol{\mem}{\mathord}{hebrewletters}{109}
\DeclareMathSymbol{\resh}{\mathord}{hebrewletters}{114}
\DeclareMathSymbol{\he}{\mathord}{hebrewletters}{104}
\DeclareMathSymbol{\qof}{\mathord}{hebrewletters}{113}
\DeclareMathSymbol{\bet}{\mathord}{hebrewletters}{98}

\usepackage{multicol, caption}

\usepackage[foot]{amsaddr}
\begin{document}

\title[Sensitivity analysis of Wasserstein DRO problems]{Sensitivity analysis of Wasserstein distributionally robust optimization problems}
\date{\today}

\author{Daniel Bartl$^1$}
\address{$^1$Department of Mathematics, University of Vienna}
\author{Samuel Drapeau$^2$}
\address{$^2$School of Mathematical Sciences and Shanghai Advanced Institute of Finance (SAIF/CAFR), Shanghai Jiao Tong University}
\author{Jan Ob{\l}{\'o}j$^3$}
\author{Johannes Wiesel$^3$}
\address{$^3$Department of Statistics, Columbia University}

\thanks{Support from the European Research Council under the EU's $7^{\text{th}}$ FP / ERC grant agreement no. 335421, the Vienna Science and Technology Fund (WWTF) project MA16-021 and the Austrian Science Fund (FWF) project P28661 as well as the National Science Foundation of China, Grant Numbers 11971310 and 11671257, Shanghai Jiao Tong University, Grant ``Assessment of Risk
and Uncertainty in Finance'' number AF0710020 are gratefully acknowledged. We thank Jose Blanchet, Peyman Mohajerin Esfahani, Daniel Kuhn and Mike Giles for helpful comments on an earlier version of the paper.
}

%\begin{abstract}
%
%\end{abstract}

\maketitle

%\addtocounter{theorem}{4}
%\addtocounter{equation}{10}

\begin{abstract}
We consider sensitivity of a generic stochastic optimization problem to model uncertainty. We take a non-parametric approach and capture model uncertainty using Wasserstein balls around the postulated model. We provide explicit formulae for the first order correction to both the value function and the optimizer and further extend our results to optimization under linear constraints. 
We present applications to statistics, machine learning, mathematical finance and uncertainty quantification. In particular, we provide explicit first-order approximation for square-root LASSO regression coefficients and deduce coefficient shrinkage compared to the ordinary least squares regression. We consider robustness of call option pricing and deduce a new Black-Scholes sensitivity, a non-parametric version of the so-called Vega. We also compute sensitivities of optimized certainty equivalents in finance and propose measures to quantify robustness of neural networks to adversarial examples. \end{abstract}

\date{\today}
\maketitle

\section{Introduction}
We consider a generic stochastic optimization problem
\begin{equation}\label{eq:stochoptim}
 \inf_{a\in\cA}\int_{\cS} f\left(x,a \right)\,\mu(dx),  
\end{equation}
where $\cA$ is the set of actions or choices, $f$ is the loss function and $\mu$ is a probability measure over the space of states $\cS$. Such problems are found across the whole of applied mathematics. The measure $\mu$ is the crucial input and it could represent, e.g., a dynamic model of the system, as is often in mathematical finance or mathematical biology, or the empirical measure of observed data points, or the training set, as is the case in statistics and machine learning applications. In virtually all the cases, there is a certain degree of uncertainty around the choice of $\mu$ coming from modelling choices and simplifications, incomplete information, data errors, finite sample error, etc. It is thus very important to understand the influence of changes in $\mu$ on \eqref{eq:stochoptim}, both on its value and on its optimizer. 
Often, the choice of $\mu$ is done in two stages: first a parametric family of models is adopted and then the values of the parameters are fixed. Sensitivity analysis of \eqref{eq:stochoptim} with changing parameters is a classical topic explored in parametric programming and statistical inference, e.g.,  \cite{Armacost:1974ig,Vogel:2007jr,Bonnans:2013vx}. It also underscores a lot of progress in the field of uncertainty quantification, see \cite{UQHandbook}. Considering $\mu$ as an abstract parameter, the mathematical programming literature looked into qualitative and quantitative stability of \eqref{eq:stochoptim}. We refer to \cite{Dupacova:1990by,Romisch:2003fe} and the references therein. When $\mu$ represents data samples, there has been a considerable interest in the optimization community in designing algorithms which are robust and, in particular, do not require excessive hypertuning, see \cite{Asi:2019ep} and the references therein. 

A more systematic approach to model uncertainty in \eqref{eq:stochoptim} is offered by the distributionally robust optimization problem
\begin{align}
\label{eq. minmax}
V(\delta)\coloneq\inf_{a\in\cA}V(\delta,a)\coloneq \inf_{a\in\cA} \sup_{\nu \in B_{\delta}\left( \mu \right)}\int_{\mathcal{S}} f\left( x,a \right)\,\nu(dx),
\end{align}
where $B_{\delta}\left( \mu \right)$ is a ball of radius $\delta$ around $\mu$ in the space of probability measures, as specified below. Such problems greatly generalize more classical robust optimization and have been studied extensively in operations research and machine learning in particular, we refer the reader to \cite{Rahimian:2019vy} and the references therein. Our goal in this paper is to understand the behaviour of these problems for small $\delta$. 
Our main results compute first-order behaviour of $V(\delta)$ and its optimizer for small $\delta$. This offers a measure of sensitivity to errors in model choice and/or specification as well as points in the abstract direction, in the space of models, in which the change is most pronounced. 
We use examples to show that our results can be applied across a wide spectrum of science.

This paper is organised as follows. We first present the main results and then, in section \ref{sec:applications}, explore their applications. Further discussion of our results and the related literature is found in section \ref{sec:lit}, which is then followed by the proofs. Online appendix \cite{SI} contains many supplementary results and remarks, as well as some more technical arguments from the proofs.

\section{Main results}\label{sec:main}
Take $d,k\in \N$, endow $\R^d$ with the Euclidean norm $|\cdot|$ and write $\interior{\Gamma}$ for the interior of a set $\Gamma$. 
Assume that $\mathcal{S}$ is a closed convex subset of $\mathbb{R}^d$.
Let $\mathcal{P}(\mathcal{S})$ denote the set of all (Borel) probability measures on $\mathcal{S}$. 
Further fix a seminorm $\|\cdot\|$ on $\mathbb{R}^d$ and denote by $\|\cdot\|_\ast$ its (extended) dual norm, i.e., $\|y\|_\ast:=\sup\{ \langle x,y\rangle : \|x\|\leq 1\}$. In particular, for $\|\cdot\|=|\cdot|$ we also have $\|\cdot\|_\ast=|\cdot|$. For $\mu,\nu\in \mathcal{P}(\mathcal{S})$, we define the $p$-Wasserstein distance as 
\begin{equation*}
    W_p(\mu, \nu)=\inf\left\{\int_{\mathcal{S}\times\mathcal{S}} \|x-y\|_\ast^p\,\pi(dx,dy)\colon \pi \in \mathrm{Cpl}(\mu,\nu) \right\}^{1/p},
\end{equation*}
where $\mathrm{Cpl}(\mu,\nu)$ is the set of all probability measures $\pi$ on $\mathcal{S}\times\mathcal{S}$ with first marginal $\pi_1:=\pi(\cdot\times\mathcal{S})=\mu$ and second marginal $\pi_2:=\pi(\mathcal{S}\times\cdot)=\nu$.
Denote the Wasserstein ball 
\begin{equation*}
    B_{\delta}(\mu) = \left\{\nu \in \mathcal{P}(\mathcal{S}) :  W_p(\mu,\nu) \leq \delta  \right\}
\end{equation*}
of size $\delta\geq0$ around $\mu$.
Note that, taking a suitable probability space $(\Omega, \bF, \bP)$ and a random variable $X\sim \mu$, we have the following probabilistic representation of $V(\delta,a)$:
 \[ \sup_{\nu \in B_{\delta}\left( \mu \right)}\int_{\mathcal{S}} f\left( x,a \right)\,\nu(dx) 
 = \sup_{Z} \E_{\bP}[f(X+Z,a)]
\] 
where the supremum is taken over all $Z$ satisfying $X+Z\in\mathcal{S}$ almost surely and $\E_{\bP}[ \|Z\|_\ast^p]\leq \delta^p$.
Wasserstein distances and optimal transport techniques have proved to be powerful and versatile tools in a multitude of applications, from economics \cite{Chiappori:10,CarlierEkland:10} to image recognition \cite{Peyre:2019dl}. The idea to use Wasserstein balls to represent model uncertainty was pioneered in 
\cite{Pflug:2007hr} in the context of investment problems. 
When sampling from a measure with a finite $p^{\textrm{th}}$ moment, the measures converge to the true distribution and Wasserstein balls around the empirical measures have the interpretation of confidence sets, see \cite{Fournier:2014kk}. In this setup, the radius $\delta$ can then be chosen as a function of a given confidence level $\alpha$ and the sample size $N$. This yields finite samples guarantees and asymptotic consistency, see \cite{MohajerinEsfahani:2018hd,OblojWieselAOS}, and justifies the use of the Wasserstein metric to capture model uncertainty. The value $V(\delta,a)$ in \eqref{eq. minmax} has a  dual representation, see \cite{gao2016distributionally,blanchet2019quantifying}, which has led to significant new developments in distributionally robust optimization, e.g., \cite{MohajerinEsfahani:2018hd,blanchet2019robust,Kuhn:2019vu, shafieezadeh2019regularization}. \\
Naturally, other choices for the distance on the space of measures are also possible: such as the Kullblack-Leibler divergence, see \cite{lam2016robust} for general sensitivity results and \cite{Calafiore:07} for applications in portfolio optimization, or the Hellinger distance, see \cite{lindsay1994efficiency} for a statistical robustness analysis. We refer to section \ref{sec:lit}
for a more detailed analysis of the state of the art in these fields.  Both of these approaches have good analytic properties and often lead to theoretically appealing closed-form solutions. However, they are also very restrictive since any measure in the neighborhood of $\mu$ has to be absolutely continuous with respect to $\mu$. In particular, if $\mu$ is the empirical measure of $N$ observations then measures in its neighborhood have to be supported on those fixed $N$ points. To obtain meaningful results it is thus necessary to impose additional structural assumptions, which are often hard to justify solely on the basis of the data at hand and, equally importantly, create another layer of model uncertainty themselves. We refer to \cite[sec.\ 1.1]{gao2016distributionally} for further discussion of potential issues with $\phi$-divergences. The Wasserstein distance, while harder to handle analytically, is more versatile and does not require any such additional assumptions.

Throughout the paper we take the convention that continuity and closure are understood w.r.t.\ $|\cdot|$. 
We assume that $\cA\subset \R^k$ is convex and closed and that the seminorm $\|\cdot\|$ is strictly convex in the sense that for two elements $x,y\in\mathbb{R}^d$ with $||x\|=\|y\|=1$ and $\|x-y\|\neq0$, we have $\|\frac{1}{2}x + \frac{1}{2}y\|<1$ (note that this is satisfied for every $l^s$-norm $|x|_s:=(\sum_{i=1}^d |x_i|^s)^{1/s}$ for $s>1$). 
We fix $p\in(1,\infty)$, let $q:=p/(p-1)$ so that $1/p+1/q=1$, and fix $\mu\in \cP(\cS)$ such that the boundary of $\cS\subset \R^d$ has $\mu$--zero measure and 
$\int_{\cS} |x|^p\,\mu(dx)<\infty$. Denote by $\Aoptim{\delta}$ the set of optimizers for $V(\delta)$ in \eqref{eq. minmax}. 

\begin{assumption} \label{ass:main} 
    The loss function $f\colon \cS\times \cA \to \R$ satisfies
	\begin{itemize}
	\item $x\mapsto f(x,a)$ is differentiable on $\interior{\mathcal{S}}$ for every $a\in\cA$.
	Moreover $(x,a)\mapsto \nabla_x f(x,a)$ is continuous and for every $r>0$ there is $c>0$ such that $|\nabla_x f(x,a)|\leq c(1+|x|^{p-1})$ for all $x\in\cS$ and $a\in \cA$ with $|a|\leq r$.
	\item For all $\delta\ge 0$ sufficiently small we have $\Aoptim{\delta}\neq\emptyset$ and for every sequence $(\delta_n)_{n\in \N}$ such that $\lim_{n\to \infty} \delta_n=0$ and $(a^\opt_n)_{n\in \N}$ such that $a^\opt_n\in \Aoptim{\delta_n}$ for all $n\in \N$ there is a subsequence which converges to some $a^\opt\in \Aoptim{0}$.
	\end{itemize}
\end{assumption}

The above assumption is not restrictive: the first part merely ensures existence of $\|\nabla_x f(\cdot, a^\opt)\|_{L^q(\mu)}$, while the second part is satisfied as soon as either $\mathcal{A}$ is compact or $V(0, \cdot)$ is coercive, which is the case in most examples of interest, see \cite[Lemma~ \ref{lem:coercive}]{SI} for further comments.

\begin{theorem}\label{thm:main}
	If Assumption \ref{ass:main} holds then $V'(0)$ is given by
	\[
	\Upsilon:=\lim_{\delta\to 0} \frac{V(\delta)-V(0)}{\delta}
	=\inf_{a^\opt \in \Aoptim{0}}\Big(\int_{\mathcal{S}} \|\nabla_x f(x,a^\opt)\|^q\,\mu(dx)\Big)^{1/q} .\]
\end{theorem}

\begin{remark}
Inspecting the proof, defining $$\tilde{V}(\delta)= \inf_{a^\opt \in \Aoptim{0}} \sup_{\nu \in B_{\delta}\left( \mu \right)}\int_{\mathcal{S}} f\left( x,a^\opt \right)\,\nu(dx) $$
we obtain $\tilde{V}'(0)=V'(0)$.  This means that for small $\delta>0$ there is no first-order gain from optimizing over all $a\in \cA$ in the definition of $V(\delta)$ when compared with restricting simply to $a^\opt\in \Aoptim{0}$, as in $\tilde{V}(\delta)$.
\end{remark}
\noindent The above result naturally extends to computing sensitivities of robust problems, i.e., $V'(r)$, see \cite[Corollary \ref{Cor:sensitivity}]{SI}, as well as to the case of stochastic optimization under linear constraints, see \cite[Theorem \ref{thm.constraints}]{SI}. We recall that $V(0,a)=\int_{\cS} f(x,a)\,\mu(dx)$.
\begin{assumption} \label{ass:sens}
Suppose the $f$ is twice continuously differentiable, $a^\opt\in \Aoptim{0}\cap \interior{\mathcal{A}}$ and
\begin{itemize}
\item $\sum_{i=1}^k |\nabla_{a_i}\nabla_x f(x,a)|\leq c(1+|x|^{p-1-\varepsilon})$ for some $\varepsilon>0$, $c>0$, all $x\in\mathbb{R}^d$ and all $a$ close to $a^\opt$.
\item The function $a\mapsto V(0,a)$ is twice continuously differentiable in the neighbourhood of $a^\opt$ and the matrix $\nabla^2_{a} V(0,a^\opt)$ is invertible.
\end{itemize}
 \end{assumption}
 
\begin{theorem}\label{thm:sens} 
Suppose $a^\opt\in \Aoptim{0}$ and $a^\opt_\delta\in \Aoptim{\delta}$ such that $a^\opt_\delta\to a^\opt$ as $\delta\to 0$ and Assumptions \ref{ass:main} and \ref{ass:sens} are satisfied. 
Then, if $\nabla_x f(x,a^\opt)\neq 0$ or  $\nabla_x\nabla_a f(x,a^\opt)= 0$  $\mu$-a.e., 
\begin{align*}
\beth&:= \lim_{\delta \to 0}\frac{a^\opt_{\delta}-a^\opt}{\delta}=-\Big(\int_{\mathcal{S}} \|\nabla_x f(x,a^\opt)\|^q\,\mu(dx)\Big)^{\frac{1}{q}-1} \\
&\quad\ \cdot (\nabla^2_a V(0,a^\opt))^{-1} \int_{\mathcal{S}} \frac{\nabla_{x}\nabla_a f(x,a^\opt)\, h(\nabla_x f(x,a^\opt))}{\|\nabla_x f(x,a^\opt)\|^{1-q}} \, \mu(dx),
\end{align*}
where $h\colon \mathbb{R}^d\setminus\{0\}\to \{x\in \R^d \ : \ \|x\|_\ast=1\}$ is the unique function satisfying $\langle \cdot, h(\cdot) \rangle =\|\cdot\|$, see \cite[Lemma \ref{dani:2}]{SI}. In particular, $h(\cdot)=\cdot/|\cdot|$ if $\|\cdot\|=|\cdot|$.
\end{theorem}
Above and throughout the convention is that $\nabla_x f(x,a)\in \R^{d\times 1},$ $\nabla_{a_i}\nabla_x f(x,a)\in \R^{d\times 1}$, $\nabla_a f(x,a)\in \R^{k\times 1}$, $\nabla_x \nabla_a f(x,a)\in \R^{k\times d}$ and   $0/0=0$. The assumed existence and convergence of optimizers holds, e.g., with suitable convexity of $f$ in $a$, see \cite[Lemma \ref{lem:optimal.strategies.converge}]{SI} for a worked out setting. In line with the financial economics practice, we gave our sensitivities letter symbols, $\Upsilon$ and $\beth$, loosely motivated by $\Upsilon \pi \acute{o} \delta \epsilon \iota \gamma \mu \alpha$, the Greek for \emph{Model}, and $\he\resh\qof\bet$, the Hebrew for \emph{control}.

\section{Applications}\label{sec:applications}
We now illustrate the universality of Theorems \ref{thm:main} and \ref{thm:sens} by considering their applications in a number of different fields. 
Unless otherwise stated, $\cS=\R^d$, $\cA=\R^k$ and $\int$ means $\int_\cS$. 

\subsection{Financial Economics} 
We start with the simple example of risk-neutral pricing of a call option written on an underlying asset $(S_t)_{t\leq T}$. Here, $T,K>0$ are the maturity and the strike respectively, $f(x,a)=(S_0x-K)^+$ and $\mu$ is the distribution of $S_T/S_0$. We set interest rates and dividends to zero for simplicity. In the \cite{BlackScholes:73} model, $\mu$ is a log-normal distribution, i.e., $\log(S_T/S_0)\sim \mathcal{N}(-\sigma^2T/2,\sigma^2 T)$ is Gaussian with mean $-\sigma^2T/2$ and variance $\sigma^2T$. In this case, $V(0)$ is given by the celebrated Black-Scholes formula. Note that this example is particularly simple since $f$ is independent of $a$. However, to ensure risk-neutral pricing, we have to impose a linear constraint on the measures in $B_\delta(\mu)$, giving 
                \begin{equation}\label{eq:robustcallprice}
                    \sup_{\nu \in B_\delta(\mu)\colon \int_{}^{} x\nu(dx) = 1} \int_{}^{} (S_0x-K)^+ \nu(dx).
                \end{equation}
To compute its sensitivity we encode the constraint using a Lagrangian and apply Theorem \ref{thm:main}, see \cite[Rem. \ref{rk:abscont}, Thm.~\ref{thm.constraints}]{SI}. For $p=2$, letting $k=K/S_0$ and $\mu_k=\mu([k,\infty))$, the resulting formula, see \cite[Example \ref{ex:martingale}]{SI}, is given by
\begin{align*}
\Upsilon &= S_0 \sqrt{\int \Big(\mathbf{1}_{x\geq k}-\mu_k\Big)^2\mu(dx)}=S_0\sqrt{\mu_k(1-\mu_k)}.
\end{align*}
Let us specialise to the log-normal distribution of the Black-Scholes model above and denote the quantity in \eqref{eq:robustcallprice} as $\mathcal{R}BS(\delta)$. It may be computed exactly using methods from \cite{Bartl:2019hk} and Figure \ref{fig:BS} compares the exact value and the first order approximation. 
\begin{figure}%[tbhp]
\centering
\includegraphics[width=.8\linewidth]{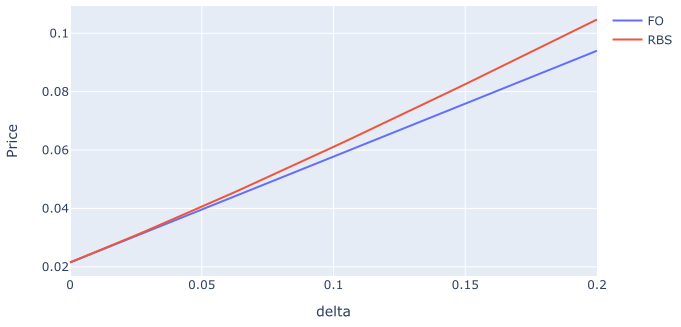}
\caption{DRO value $\mathcal{R}$BS$(\delta)$ vs the first order (FO) approximation $\mathcal{R}$BS$(0)+\Upsilon \delta$, $S_0=T=1$, $K=1.2$, $\sigma=0.2$.}
\label{fig:BS}
\end{figure}
We have $\Upsilon =S_0\sqrt{\Phi(d_-)(1-\Phi(d_-))}$, where $d_-=\frac{\log(S_0/K)-\sigma^2T/2}{\sigma\sqrt{T}}$ and $\Phi$ is the cdf of $\mathcal{N}(0,1)$ distribution. It is also insightful to compare $\Upsilon$ with a parametric sensitivity. If instead of Wasserstein balls, we consider $\{\mathcal{N}(-\tilde \sigma^2T/2,\tilde\sigma^2 T): |\sigma-\tilde\sigma|\leq \delta\}$ the resulting sensitivity is known as the Black-Scholes Vega and given by $\mathcal{V}=S_0\Phi'(d_-+\sigma\sqrt{T})$. We plot the two sensitivities in Figure \ref{fig:BSsens}. 
It is remarkable how, for the range of strikes of interest, the non-parametric model sensitivity $\Upsilon$ traces out the usual shape of $\mathcal{V}$ but shifted upwards to account for the idiosyncratic risk of departure from the log-normal family. More generally, given a book of options with payoff $f=f^+-f^-$ at time $T$, with $f^+,f^-\geq 0$, we could say that the book is $\Upsilon$-neutral if the sensitivity $\Upsilon$ was the same for $f^+$ and for $f^-$. In analogy to Delta-Vega hedging standard, one could develop a non-parametric model-agnostic Delta-Upsilon hedging. We believe these ideas offer potential for exciting industrial applications and we leave them to further research.
\begin{figure}%[tbhp]
\centering
\includegraphics[width=.8\linewidth]{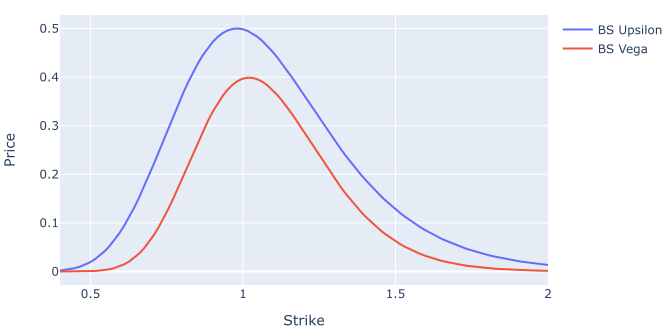}
\caption{Black-Scholes model: $\Upsilon$ vs $\mathcal{V}$, $S_0=T=1$, $\sigma=0.2$.}
\label{fig:BSsens}
\end{figure}

We turn now to the classical notion of optimized Certainty Equivalent (OCE) of \cite{BenTal:jh}. It is a decision theoretic criterion designed to split a liability between today and tomorrow's payments. It is also a convex risk measure in the sense of \cite{ADEH:99} and covers many of the popular risk measures such as expected shortfall or entropic risk, see \cite{BenTal:2007fo}. We fix a convex monotone function $l\colon\mathbb{R}\to\mathbb{R}$ which is bounded from below and $g\colon\mathbb{R}^d\to\mathbb{R}$. Here $g$ represent the payoff of a financial position and $l$ is the negative of a utility function, or a loss function. 
We take $\| \cdot\|=|\cdot|$ and refer to \cite[Lemma \ref{lem:optimal.strategies.converge}]{SI} for generic sufficient conditions for Assumptions \ref{ass:main} and \ref{ass:sens} to hold in this setup. The OCE corresponds to $V$ in \eqref{eq:stochoptim} for $f(x,a)=l(g(x)-a)+a$ and $\cA=\R$, $\cS=\R^d$. Theorems \ref{thm:main} and \ref{thm:sens} yield the sensitivities
\begin{align*}
&\Upsilon = \inf_{a^\opt\in \Aoptim{0}}\left( \int \big| l'\big( g(x)-a^\opt\big) \nabla g(x) \big|^q\,\mu(dx)\right)^{1/q},  \\
&\beth = \Big(\int |l'(g(x)-a^\opt)\,\nabla g(x)|^2\,\mu(dz)\Big)^{-1/2} \\
&\qquad\qquad \cdot \frac{\int l''(g(x)-a^\opt)\,l'(g(x)-a^\opt)\,(\nabla g(x))^2 \, \mu(dx)}{\int l''(g(x)-a^\opt)\,\mu(dx)},
\end{align*}
where, for simplicity, we took $p=q=2$ for the latter.
\\ A related problem considers hedging strategies which minimise the expected loss of the hedged position, i.e., $f(x,a)=l\left( g(x)+\langle a,x-x_0\rangle \right)$, where $\cA=\R^k$ and $(x_0,x)$ represent today and tomorrow's traded prices. We compute $\Upsilon$ as 
 \begin{align*}
 \inf_{a^\opt\in \Aoptim{0}} \left( \int \big| l'\big( g(x)+\langle a^\opt,x-x_0\rangle \big) (\nabla  g(x)+a^\opt) \big|^q\,\mu(dx)\right)^{1/q}.
    \end{align*}
Further, we can combine loss minimization with OCE and consider $a=(H,m)\in \R^k\times \R$, $f(x,(h,m))=l ( g(x)+\langle H,x-x_0\rangle+m)-m$. This gives $V'(0)$ as the infimum over $(H^\opt,m^\opt)\in \Aoptim{0}$ of 
$$ \left( \int \big| l'\big( g(x)+\langle H^\opt, x-x_0 \rangle +m^\opt\big) \left( \nabla g(x) +H^\opt\right)\big|^q\,\mu(dx)\right)^{1/q}.$$
The above formulae capture non-parametric sensitivity to model uncertainty for examples of key risk measurements in financial economics. 
To the best of our knowledge this has not been achieved before. 

Finally, we consider briefly the classical mean-variance optimization of \cite{Markowitz:52}. Here $\mu$ represents the loss distribution across the assets and $a\in \R^d$, $\sum_{i=1}^d a_i=1$ are the relative investment weights. The original problem is to minimise the sum of the expectation and $\gamma$ standard deviations of returns $\langle a, X\rangle$, with $X\sim \mu$. Using the ideas in \cite[Example 2]{PPW:12} and considering measures on $\R^d\times \R^d$, we can recast the problem as \eqref{eq:stochoptim}. Whilst \cite{PPW:12} focused on the asymptotic regime $\delta \to \infty$, their non-asymptotic statements are related to our Theorem \ref{thm:main} and either result could be used here to obtain that $V(\delta)\approx V(0)+\sqrt{1-\gamma^2}\delta$. 

\subsection{Neural Networks} We specialise now to quantifying robustness of neural networks (NN) to adversarial examples. This has been an important topic in machine learning since \cite{szegedy2013intriguing} observed that NN consistently misclassify inputs formed by applying small worst-case perturbations to a data set. This produced a number of works offering either explanations for these effects or algorithms to create such adversarial examples, e.g. \cite{goodfellow2014explaining,li2019robustra, carlini2017towards, wong2017provable,weng2018evaluating,araujo2019robust, mangal2019robustness} to name just a few. The main focus of research works in this area, see \cite{bastani2016measuring}, has been on faster algorithms for finding adversarial examples, typically leading to an overfit to these examples without any significant generalisation properties. The viewpoint has been mainly pointwise, e.g., \cite{szegedy2013intriguing}, with some generalisations to probabilistic robustness, e.g., \cite{mangal2019robustness}.

In contrast, we propose a simple metric for measuring robustness of neural networks which is independent of the architecture employed and the algorithms for identifying adversarial examples. In fact, Theorem \ref{thm:main} offers a simple and intuitive way to formalise robustness of neural networks: for simplicity consider a $1$-layer neural network trained on a given distribution $\mu$ of pairs $(x,y)$, i.e. $(A^\opt_1, A_2^\opt, b_1^\opt, b_2^\opt)$ solve
\begin{align*}
\inf \int &|y-\left((A_2(\cdot)+b_2)\circ \sigma \circ (A_1(\cdot)+b_1)\right)(x)|^p\,\mu(dx, dy),
\end{align*}
where the $\inf$ is taken over $a=(A_1,A_2,b_1,b_2)\in \cA= \R^{k\times d}\times \mathbb{R}^{d\times k}\times \R^k\times\R^d$, 
for a given activation function $\sigma:\R\to \R$, where the composition above is understood componentwise.
Set $f(x,y; A, b):= |y-(A_2(\cdot)+b_2)\circ \sigma \circ (A_1(\cdot)+b_1)(x)|^p$. Data perturbations are captured by $\nu\in B_{\delta}^p(\mu)$ and \eqref{eq. minmax} offers a robust training procedure. 
The first order quantification of the NN sensitivity to adversarial data is then given by
\begin{align*}
\left(\int |\nabla f(x,y; A^\opt, b^\opt)|^q\, \mu(dx,dy)\right)^{1/q}.
\end{align*}
A similar viewpoint, capturing robustness to adversarial examples through optimal transport lens, has been recently adopted by other authors. The dual formulation of \eqref{eq. minmax} was used by  \cite{shafieezadeh2019regularization} to reduce the training of neural networks to tractable linear programs.  \cite{sinha2017certifying} modified \eqref{eq. minmax} to consider a penalised problem $\inf_{a\in\cA} \sup_{\nu \in \cP(\cS)}\int_{\mathcal{S}} f\left( x,a \right)\,\nu(dx)- \gamma W_p(\mu,\nu)$ to propose new stochastic gradient descent algorithms with inbuilt robustness to adversarial data. 

\subsection{Uncertainty Quantification} 
In the context of UQ the measure $\mu$ represents input parameters of a (possibly complicated) operation $G$ in a physical, engineering or economic system. We consider the so-called \textit{reliability} or \emph{certification problem}: for a given set $E$ of undesirable outcomes, one wants to control $\sup_{\nu \in \cP} \nu(G(x)\in E)$, for a set of probability measures $\cP$. The distributionally robust adversarial classification problem considered recently by \cite{ho2020adversarial} is also of this form, with Wasserstein balls $\cP$ around an empirical measure of $N$ samples. Using the dual formulation of \cite{blanchet2019quantifying}, they linked the problem to minimization of the conditional value-at-risk and proposed a reformulation, and numerical methods, in the case of linear classification. We propose instead a regularised version of the problem and look for 
\begin{align*}
\delta(\alpha):=\sup\left\{ \delta\ge 0:\  \inf_{\nu \in B_{\delta}(\mu)} \int d(G(x),E)\, \nu(dx) \ge \alpha\right\}
\end{align*}
for a given safety level $\alpha$. 
We thus consider the average distance to the undesirable set, $d(G(x),E):=\inf_{e\in E}|G(x)-e|$, and not just its probability. The quantity $\delta(\alpha)$ could then be used to quantify the implicit uncertainty of the certification problem, where higher $\delta$ corresponds to less uncertainty. Taking statistical confidence bounds of the empirical measure in Wasserstein distance into account, see \cite{Fournier:2014kk}, $\delta$ would then determine the minimum number of samples needed to estimate the empirical measure. 

Assume that $E$ is convex.
Then $x\mapsto d(x,E)$ differentiable everywhere except at the boundary of $E$ with $\nabla_x d(x,E)=0$ for $x\in \interior{E}$ and $|\nabla_x d(x,E)|=1$ for all $x\in \bar{E}^c$. Further, assume $\mu$ is absolutely continuous w.r.t.\ Lebesgue measure on $\cS$. Theorem \ref{thm:main}, using \cite[Remark \ref{rk:abscont}]{SI}, gives a first-order expansion for the above problem
\begin{align*}
&\inf_{\nu\in B_{\delta}(\mu)} \int  d(G(x),E) \, \nu(dx)= \int  d(G(x),E) \, \mu(dx)\\
&\qquad -\left(\int |\nabla_x d(G(x),E)\nabla_x G(x) |^q\,\mu(dx)\right)^{1/q} \delta +o(\delta).
\end{align*}
In the special case $\nabla_x G(x)=cI$ this simplifies to 
\begin{align*}
\int  d(G(x),E) \, \mu(dx)- c\left( \mu(G(x) \notin E)\right)^{1/q}\delta+o(\delta)
\end{align*}
and the minimal measure $\nu$ pushes every point $G(x)$ not contained in $E$ in the direction of the orthogonal projection. This recovers the intuition of \cite[Theorem 1]{chen2018data}, which in turn rely on \cite[Corollary 2, Example 7]{gao2016distributionally}. Note however that our result holds for general measures $\mu$. We also note that such an approximation could provide an ansatz for dimension reduction, by identifying the dimensions for which the partial derivatives are negligible and then projecting $G$ on to the corresponding lower-dimensional subspace (thus providing a simpler surrogate for $G$). This would be an alternative to a basis expansion (e.g., orthogonal polynomials) used in UQ and would exploit the interplay of properties of $G$ and $\mu$ simultaneously. 

\subsection{Statistics} We discuss two applications of our results in the realm of statistics. 
We start by highlighting the link between our results and the so-called \emph{influence curves} (IC) in robust statistics.
 For a functional $\mu\mapsto T(\mu)$ its IC is defined as
\begin{align*}
\mathrm{IC}(y)=\lim_{t\to 0} \frac{T(t\delta_y +(1-t)\mu)-T(\mu)}{t}.
\end{align*}
Computing the $\mathrm{IC}$, if it exists, is in general hard and closed form solutions may be unachievable. However, for the so-called M-estimators, defined as optimizers for $V(0)$, 
\begin{align*}
T(\mu):=\mathrm{argmin}_a \int f(x,a)\mu(dx), 
\end{align*}
for some $f$ (e.g., $f(x,a)=|x-a|$ for the median), we have 
\begin{align*}
\mathrm{IC}(y)= \frac{\nabla_a f(y,T(\mu))}{-\int \nabla^2_{a}f(s,T(\mu))\,\mu(ds)},
\end{align*}
under suitable assumptions on $f$, see \cite[section 3.2.1]{huber1981robust}. In comparison, writing $T^\delta$ for the optimizer for $V(\delta)$, Theorem \ref{thm:sens} yields 
\begin{align}\label{eq:ex_robust}
\lim_{\delta \to 0}\frac{T^{\delta}-T(\mu)}{\delta}=  \frac{\int\nabla_{x}\nabla_a f(x,T(\mu)) \nabla_x f(x,T(\mu))\, \mu(dx)}{-\int \nabla^2_a f(s,T(\mu))\,\mu(ds)},
\end{align}
under Assumption \ref{ass:sens} and normalisation $\|\nabla_x f(x,T(\mu))\|_{L^p(\mu)}=1$. To investigate the connection let us Taylor-expand $\mathrm{IC}(y)$ around $x$ to obtain
\begin{align*}
\mathrm{IC}(y)- \mathrm{IC}(x)=& \frac{\nabla_{a}\nabla_x f(x, T(\mu))}{-\int \nabla^2_{a}f(s,T(\mu))\,\mu(ds)} (y-x).
\end{align*}
Choosing $y=x+\delta \nabla f_x(x,T(\mu))$ and integrating both sides over $\mu$ and dividing by $\delta$, we obtain the asymptotic equality
\begin{align*}
 \int \frac{\mathrm{IC}(x+\delta\nabla_x f(x, T(\mu)))- \mathrm{IC}(x)}{\delta} \, \mu(dx)\approx \frac{T^{\delta}-T(\mu)}{\delta}
\end{align*}
for $\delta \to 0$ by \eqref{eq:ex_robust}. We conclude that considering the average directional derivative of IC in the direction of $\nabla f_x(x,T(\mu))$ gives our first-order sensitivity. For an interesting conjecture regarding the comparison of influence functions and sensitivities in KL-divergence we refer to \cite[Section 7.3]{lam2018sensitivity} and \cite[Section 3.4.2]{lam2016robust}.

Our second application in statistics exploits the representation of the LASSO/Ridge regressions as robust versions the standard linear regression. We consider $\cA=\R^k$ and $\cS=\R^{k+1}$. If instead of the Euclidean metric we take $\|(x,y)\|_{\ast}=|x|_r\mathbf{1}_{\{y=0\}}+\infty \mathbf{1}_{\{y\neq 0\}}$, for some $r>1$ and $(x,y)\in \R^{k}\times \R$, in the definition of the Wasserstein distance, then \cite{blanchet2019robust} showed that 
\begin{equation}
\label{eq:lasso}
\begin{split}
&\inf_{a\in \R^k}\sup_{\nu \in B_{\delta}\left( \mu \right)}\int (y-\langle x, a \rangle )^2\,\nu(dx,dy)\\
&\qquad = \inf_{a\in \R^k} \left(\sqrt{\int (y-\langle a,x \rangle)^2\,\mu(dx,dy)} +\delta |a|_s \right)^{2}
\end{split}
\end{equation}
holds, where $1/r+1/s=1$. The $\delta=0$ case is the ordinary least squares regression. For $\delta>0$, the RHS for $s=2$ is directly related to the Ridge regression, while the limiting case $s=1$ is called the square-root LASSO regression, a regularised variant of linear regression well known for its good empirical performance. Closed-form solutions to \eqref{eq:lasso} do not exist in general and it is a common practice to use numerical routines to solve it  approximately. Theorem \ref{thm:sens} offers instead an explicit first-order approximation of $a^\opt_\delta$ for small $\delta$. We denote by $a^\opt$ the ordinary least squares estimator and by $I$ the $k\times k$ identity matrix. Note that the first order condition on $a^\opt$ implies that $\int (y-\langle a^\opt,x\rangle )x_i\mu(dx,dy)=0$ for all $1\leq i\leq k$. In particular, $V(0)=\int (y^2-\langle a^\opt, x\rangle y)\mu(dx,dy)$ and $a^\opt = D^{-1}\int yx \mu(dx,dy)$, where we assume the system is overdetermined so that $D=\int xx^T \, \mu(dx,dy)$ is invertible. Letting $J=a^\opt x^T+(Ia^\opt)(Ix)$ a direct computation, see \cite{SI}, yields
\begin{align}
a^\opt_{\delta}\approx\ a^\opt-  \sqrt{V(0)} D^{-1}\, h(a^\opt) \delta. \label{eq:lasso2}
\end{align}
For $s=2$, $h(a^\opt)=a^\opt/|a^\opt|_2$ and for $s=1$, $h(a^\opt) = \text{sign}(a^\opt)$ and hence\footnote{The case $s=1$, inspecting the proof, we see that Theorem \ref{thm:sens} still holds since $a^\opt$ does not have zero components $\mu$-a.s., which are the only points of discontinuity of $h$.} $a^\opt_\delta$ is approximately 
\begin{equation}\label{eq:lasso_explicit}
\left(1-\frac{\sqrt{V(0)}}{|a^\opt |_2}D^{-1}\delta\right)a^\opt\text{ and }a^\opt-\sqrt{V(0)}D^{-1}\text{sign}(a^\opt)\delta 
\end{equation} 
respectively. This corresponds to parameter shrinkage: proportional for square-root Ridge and a shift towards zero for square-root LASSO. To the best of our knowledge these are first such results and we stress that our formulae are valid in a general context and, in particular, parameter shrinkage depends on the direction through the $D^{-1}$ factor. Figure \ref{fig:LASSO} compares the first order approximation with the actual results and shows a remarkable fit. 
\begin{figure}%[tbhp]
\centering
\includegraphics[width=.8\linewidth]{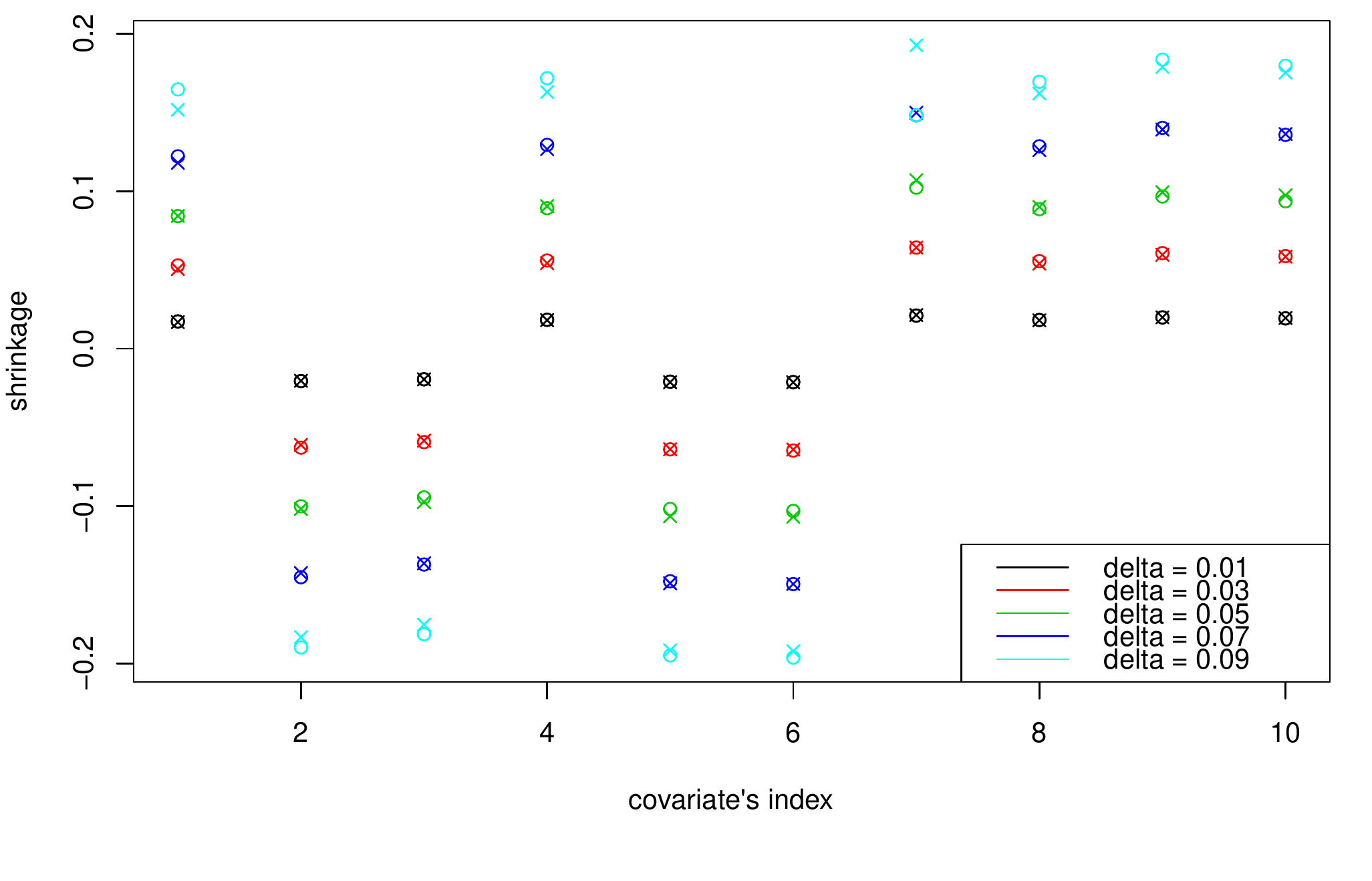}
\caption{Square-root LASSO parameter shrinkage $a^\opt_\delta-a^\opt_0$: exact (o) and the first-order approximation (x) in \eqref{eq:lasso_explicit}. $2000$ observations generated according to $Y=1.5 X_1 -3X_2 - 2X_3 + 0.3X_4- 0.5X_5-0.7X_6+0.2X_7+0.5X_8+1.2X_9+0.8 X_{10}+\epsilon$ with all $X_i,\epsilon$ i.i.d.\ $\mathcal{N}(0,1)$.}
\label{fig:LASSO}
\end{figure}
Furthermore, our results agree with what is known in the canonical test case for the (standard) Ridge and LASSO, see \cite{Tobshirani:96}, when $\mu=\mu_N$ is the empirical measure of $N$ i.i.d.\ observations, the data is centred and the covariates are orthogonal, i.e., $D=\frac{1}{N}I$. In that case, \eqref{eq:lasso_explicit} simplifies to
$$
\left(1 - \delta\sqrt{N \left(\frac{1}{R^2}-1\right)}\right)a^\opt\text{ and }a^\opt-\sqrt{N}\, |y|\,\sqrt{1-R^2}\,\text{sign}(a^\opt) \delta,
$$
where $R^2$ is the usual coefficient of determination.

The case of $\mu_N$ is naturally of particular importance in statistics and data science and we continue to consider it in the next subsection. In particular, we characterise the asymptotic distribution of $\sqrt{N}(a^\opt_{1/\sqrt{N}}-a^\opt)$, where $a^\opt_\delta\in \Aoptim{\delta}(\mu_N)$ and $a^\opt\in \Aoptim{0}(\mu_\infty)$ is the optimizer of the non-robust problem for the data-generating measure. This recovers the central limit theorem of \cite{Blanchet:2019ue}, a link we explain further in  section \ref{sec:lit}\ref{sec:CLT}.

\subsection{Out-of-sample error}\label{sec:outofsample}
A benchmark of paramount importance in optimization is the so-called \emph{out-of-sample error}, also known as the \emph{prediction error} in statistical learning.  
Consider the setup above when $\mu_N$ is the empirical measure of $N$ i.i.d.\ observations sampled from the ``true" distribution $\mu=\mu_\infty$ and take, for simplicity, $\|\cdot \|=|\cdot |_s$, with $s>1$. Our aim is to compute the optimal $a^\opt$ which solves the original problem \eqref{eq:stochoptim}. However, we only have access to the training set, encoded via $\mu_N$. Suppose we solve the distributionally robust optimization problem \eqref{eq. minmax} for $\mu_N$ and denote the robust optimizer $a^{\opt,N}_\delta$. Then the \emph{out-of-sample error} 
\begin{align*}
V(0,a^{\opt,N}_\delta)-V(0,a^\opt)=\int f(x,a^{\opt,N}_\delta)\,\mu(dx)-\int f(x,a^{\opt})\,\mu(dx)
\end{align*}
quantifies the error from using $a^{\opt,N}_\delta$ as opposed to the true optimizer $a^\opt$.

While this expression seems to be hard to compute explicitly for finite samples, Theorem \ref{thm:sens} offers a way to find the asymptotic distribution of a (suitably rescaled version of) the out-of-sample error. We suppose the assumptions in Theorem \ref{thm:sens} are satisfied and note that the first order condition for $a^\opt$ gives $\nabla_a V(0,a^\opt)=0$. Then, a second-order Taylor expansion gives
\begin{align}\label{eq:taylor}
V(0,a^{\opt,N}_\delta)-V(0,a^\opt) = \frac{1}{2} (a^{\opt,N}_\delta-a^{\opt})^T\nabla_a^2 V(0,\tilde a)(a^{\opt,N}_\delta-a^{\opt}),
\end{align}
for some $\tilde a$, (coordinate-wise) between $a^\opt$ and $a^{\opt,N}_\delta$. Now we write
\begin{align*}
a^{\opt,N}_\delta-a^{\opt}&= a^{\opt,N}_\delta-a^{\opt,N}+a^{\opt,N}-a^{\opt},
\end{align*}
where we define $a^{\opt,N}$ as the optimizer of the non-robust problem \eqref{eq:stochoptim} with $\mu$ replaced by $\mu_N$. In particular the $\delta$-method for M-estimators implies that
\begin{align}\label{eq:M-est asymptotics}
\sqrt{N}(a^{\opt,N}-a^{\opt}) \Rightarrow (\nabla_a^2 V(0, a^{\opt}))^{-1} H,
\end{align}
where $H\sim \mathcal{N}\left(0, \int (\nabla_a f(x,a^{\opt}))^T\nabla_a f(x,a^{\opt})\,\mu(dx)\right)$ and $\Rightarrow$ denotes the convergence in distribution. On the other hand, for a fixed $N\in \N$, Theorem \ref{thm:sens} applied to $\mu_N$ yields 
\begin{align}
a^{\opt,N}_\delta-a^{\opt,N} =&-\Big(\int |\nabla_x f(x,a^{\opt,N})|_s^q\,\mu_N(dx)\Big)^{\frac{1}{q}-1}\cdot \left(\int \nabla_a^2 f(x, a^{\opt,N})\,\mu_N(dx)\right)^{-1} \label{eq:Nsamples opt diff fin}\\
&\quad\ \cdot \int \frac{\nabla_{x}\nabla_a f(x,a^{\opt,N})\, h(\nabla_x f(x,a^{\opt,N}))}{|\nabla_x f(x,a^{\opt,N})|_s^{1-q}} \, \mu_N(dx)\cdot \delta+o(\delta)\nonumber \\
  =& -\left( (\nabla_a^2 V(0, a^{\opt}))^{-1}\Theta +\Delta_N \right)\cdot \delta+o(\delta),\label{eq:Nsamples opt diff}
\end{align}
where 
\begin{align*}
\Theta & := \Big(\int |\nabla_x f(x,a^\opt)|^q_s\,\mu(dx)\Big)^{\frac{1}{q}-1}\cdot 
   \int  \frac{\nabla_{x}\nabla_a f(x,a^\opt)\, h(\nabla_x f(x,a^\opt))}{|\nabla_x f(x,a^\opt)|_s^{1-q}} \, \mu(dx),\\
\Delta_N &:=\Big(\int |\nabla_x f(x,a^{\opt,N})|_s^q\,\mu_N(dx)\Big)^{\frac{1}{q}-1} \cdot \left(\int \nabla_a^2 f(x, a^{\opt,N})\,\mu_N(dx)\right)^{-1} \\
&\quad \cdot \int \frac{\nabla_{x}\nabla_a f(x,a^{\opt,N})\, h(\nabla_x f(x,a^{\opt,N}))}{|\nabla_x f(x,a^{\opt,N})|_s^{1-q}} \, \mu_N(dx)-(\nabla_a^2 V(0, a^{\opt}))^{-1}\Theta.
\end{align*}
Almost surely (w.r.t.\ sampling of $\mu_N$), we know that $\mu_N\to \mu$ in $W_p$ as $N\to \infty$, see \cite{Fournier:2014kk}, and under the regularity and growth assumptions on $f$ in \cite[eq.\ \eqref{eq:condit}]{SI} we check that $\Delta_N\to 0$ a.s., see \cite[Example \ref{ex:OutOfSample}]{SI} for details. 
In particular, taking $\delta=1/\sqrt{N}$ and combining the above with \eqref{eq:M-est asymptotics} we obtain
\begin{equation}\label{eq:CLTfromOSE}
 \sqrt{N} \left(a^{\opt,N}_{1/\sqrt{N}}-a^\opt \right) \Rightarrow (\nabla_a^2 V(0, a^\opt))^{-1}(H-\Theta).
\end{equation}
This recovers the central limit theorem of \cite{Blanchet:2019ue}, as discussed in more detail in section \ref{sec:lit}\ref{sec:CLT} below.
Together, \eqref{eq:taylor} and \eqref{eq:Nsamples opt diff} give us the a.s.\ asymptotic behaviour of the out-of-sample error
\begin{align}\label{eq:out-of-sample-error}
V(0,a^{\opt,N}_\delta)-V(0,a^\opt)= \frac{1}{2N}(H-\Theta)^T(\nabla_a^2 V(0, a^{\opt}))^{-1} (H-\Theta)+o\left(\frac{1}{N}\right).
\end{align}
These results also extend and complement \cite[Prop. 17]{anderson2019improving}. \cite{anderson2019improving} investigate when the distributionally robust optimizers $a^{\opt,N}_\delta$ yield, on average, better performance than the simple in-sample optimizer $a^{\opt, N}$. To this end, they consider the expectation, over the realisations of the empirical measure $\mu_N$ of 
\begin{align*}
V(0,a^{\opt,N}_\delta)-V(0,a^{\opt,N})= \int f(x,a^{\opt,N}_\delta)\,\mu(dx)-\int f(x,a^{\opt,N})\,\mu(dx).
\end{align*}
This is closely related to the out-of-sample error and our derivations above can be easily modified. The first order term in the Taylor expansion no longer vanishes and, instead of \eqref{eq:taylor}, we now have
\begin{align*}
V(0,a^{\opt,N}_\delta)-V(0,a^{\opt,N}) = \nabla_a V(0,a^{\opt,N})(a^{\opt,N}_\delta-a^{\opt,N})+o(|a^{\opt,N}_\delta-a^{\opt,N}|),
\end{align*}
which holds, e.g., if for any $r>0$, there exists $c>0$ such that $\sum_{i=1}^k\Big| \nabla_{a}\nabla_{a_i} f(x,a)\Big| \le c(1+|x|^{p})$ for all $x\in \cS$, $|a|\leq r$. Combined with \eqref{eq:Nsamples opt diff fin}, this gives asymptotics in small $\delta$ for a fixed $N$. For quadratic $f$ and taking $q\uparrow \infty$, we recover the result in \cite[Prop.\ 17]{anderson2019improving}, see  \cite[Example \ref{ex:OutOfSample}]{SI} for details.

\section{Further discussion and literature review}\label{sec:lit}

We start with an overview of related literature and then focus specifically on a comparison of our results with the CLT of \cite{Blanchet:2019ue} mentioned above.

\subsection{Discussion of related literature}
Let us first remark, that while Theorem 2 offers some superficial similarities to a classical maximum theorem, which is usually concerned with continuity properties of $\delta\mapsto V(\delta)$, in this work we are instead interested in the exact first derivative of the function $\delta\mapsto V(\delta)$. Indeed the convergence $\lim_{\delta \to 0} \sup_{\nu \in B_\delta(\mu)} \int f(x)\,\nu(dx)=\int f(x)\,\mu(dx)$ follows for all $f$ satisfying $f(x)\le c(1+|x|^p)$ directly from the definition of convergence in Wasserstein metric (see e.g. \cite[Def. 6.8]{villani2008optimal}). In conclusion the main issue is to quantify the rate of this convergence by calculating the first derivative $V'(\delta)$.

Our work investigates model uncertainty broadly conceived: it includes errors related to the choice of models from a particular (parametric or not) class of models as well as the mis-specification of such class altogether (or indeed, its absence). In decision theoretic literature, these aspects are sometimes referred to as model ambiguity and model mis-specification respectively, see \cite{HansenMarinacci:16}. However, seeing our main problem \eqref{eq. minmax} in decision theoretic terms is not necessarily helpful as we think of $f$ as given and not coming from some latent expected utility type of problem. In particular, our actions $a\in \cA$ are just constants. 

In our work we decided to capture the uncertainty in the specification of $\mu$ using neighborhoods in the Wassertein distance. As already mentioned, other choices are possible and have been used in past. Possibly, the most often used alternative is the relative entropy, or the Kullblack-Leibler divergence. In particular, it has been used in this context in economics, see \cite{Hansen:tu}. To the best of our knowledge, the only comparable study of sensitivities with respect to relative entropy balls is \cite{lam2016robust}, see also \cite{lam2018sensitivity} allowing for additional marginal constraints. However, this only considered the specific case $f(x,a)=f(x)$ where the reward function is independent of the action. Its main result is
\begin{align*}
\sup_{\nu \in B^{KL}_{\delta}(\mu)} \int f(x)\, \nu(dx) = \int f(x)\, \mu(dx) +\sqrt{2 \operatorname{Var}_{\mu}(f(X))} \delta+\frac{1}{3} \frac{\kappa_{3}(f(X))}{\operatorname{Var}_{\mu}(f(X))} \delta^2+O\left(\delta^{3}\right),
\end{align*}
where $ B^{KL}_{\delta}(\mu)$ is a ball of radius $\delta^2$ centred around $\mu$ in KL-divergence, $\operatorname{Var}_{\mu}(f(X))$ and $\kappa_{3}(f(X))$ denote the variance and kurtosis of $f$ under the measure $\mu$ respectively. In particular, the first order sensitivity involves the function $f$ itself. In contrast, our Theorem 2 states $V'(\delta)=(\int (f^\prime(x))^2\,\mu(dx))^{1/2}$ and involves the first derivative $f'$. In the trivial case of a point mass $\mu=\delta_x$ we recover the intuitive sensitivity $V'(\delta)=|f^\prime(x)|$, while the results of \cite{lam2016robust} do not apply for this case. We also note that \cite{lam2016robust} requires exponential moments of the function $f$ under the baseline measure $\mu$, while we only require polynomial moments. In particular in applications in econometrics (or any field in which $\mu$ typically has fat tails), the scope of application of the corresponding results might then be decisively different. We remark however, that this requirement can be substantially weakened (to the existence of polynomial moments) when replacing KL-divergences by $\alpha$-divergences, see e.g. \cite{atar2015robust,glasserman2014robust}. We expect a sensitivity analysis similar to \cite{lam2016robust} to hold in this setting. However, to the best of our knowledge no explicit results seem to available in the literature. 
\\  
To understand the relative technical difficulties and merits it is insightful to go into the details of the statements. In fact, in the case of relative entropy and the one-period setup we are considering, the exact form of the optimizing density can be determined exactly (see \cite[Proposition 3.1]{lam2016robust}) up to a one-dimensional Langrange parameter. This is well known and is the reason behind the usual elegant formulae obtained in this context. But this then reduces the problem in \cite{lam2016robust}  to a one-dimensional problem, which can be well-approximated via a Taylor approximation. In contrast, when we consider balls in the Wasserstein distance, the form of the optimizing measure is not known (apart from some degenerate cases). In fact a key insight of our results is that the optimizing measure can be approximated by a deterministic shift in the direction $(x+f^\prime(x)\delta)_*\mu$ (this is, in general, not exact but only true as a first order approximation). The reason for these contrastive starting points of the analyses is the fact that Wasserstein balls contain a more heterogeneous set of measures, while in the case of relative entropy, exponentiating $f$ will always do the trick. We remark however that this is not true for the finite-horizon problems considered in \cite[Section 3.2]{lam2016robust} any more, where the worst-case measure is found using an elaborate fixed-point equation.\\
A point which further emphasizes the fact that the topology introduced by the Wasserstein metric is less tractable is the fact that
\begin{align*}
W^p_p(\mu, \nu)=\lim_{\epsilon\to 0} \inf_{\pi\in \Pi(\mu, \nu)} \int |x-y|^p\,\pi(dx,dy)+\epsilon H(\pi \mid \mu \otimes \nu)=\lim_{\epsilon\to 0}\epsilon \inf_{\pi\in \Pi(\mu, \nu)} H(\pi \mid R^\epsilon),
\end{align*}
where $H(\pi \mid R^\epsilon) = \int \log\left( \frac{d\pi}{dR^\epsilon}\right)\, d\pi$ is the relative entropy and $$dR^\epsilon= c_0 \exp(-|x-y|^p/\epsilon)d (\mu \otimes \nu)$$ for some normalising constant $c_0>0$, see, e.g., \cite{carlier2017convergence}. This is known as the entropic optimal transport formulation and has received considerable interest in the ML community in the past years (see e.g. \cite{peyre2019computational}).  In particular, the Wasserstein distance can be approximated by relative entropy, but only with respect to reference measures on the product space. As we consider optimization over $\nu$ above it amounts to changing the reference measure. In consequence the topological structure imposed by Wasserstein distances is more intricate compared to relative entropy, but also more flexible.

The other well studied distance is the Hellinger distance. \cite{lindsay1994efficiency} calculates influence curves for the minimum Hellinger distance estimator $a^{\mathrm{Hell},\opt}$ on a countable sample space. Their main result is that for the choice $f(x,a)=\log(\ell(x,a))$ (where $(\ell(x,a))_{a\in \mathcal{A}}$ is a collection of parametric densities)
\begin{align*}
IC(x)=-\left(\nabla_a^2V(0, a^{\mathrm{Hell},\opt})\right)^{-1}\nabla_a \log(\ell(x,a^{\mathrm{Hell},\opt})),
\end{align*}
the product of the inverse Fisher information matrix and the score function, which is the same as for the classical maximum likelihood estimator. Denote by $\mu_N$ the empirical measure of $N$ data samples and by $a^{\mathrm{Hell},\opt}(N)$ the corresponding minimum Hellinger distance estimator for $\mu_N$. In particular this result implies the same CLT as for M-estimators given by
\begin{align*}
N^{1/2} (a^{\mathrm{Hell},\opt}(N)-a^{\mathrm{Hell},\opt}) \Rightarrow (\nabla_a^2 V(0, a^\opt))^{-1}  H 
\end{align*}
where $H\sim \mathcal{N}\left(0, \int \nabla_a f(x,a^{\mathrm{Hell},\opt})^T \nabla_a f(x, a^{\mathrm{Hell},\opt})\, \mu(dx)\right).$
As we discuss in the next section, our Theorem \ref{thm:sens} yields a similar CLT, namely
\begin{align*}
N^{1/2} (a^{\opt,N}_{1/\sqrt{N}}-a^\opt) \Rightarrow (\nabla_a^2 V(0, a^\opt))^{-1} \cdot \left(H -\nabla_a\sqrt{\int |\nabla_x f(x, a^\opt)|_s^2\, \mu(dx)}\,\right).
\end{align*}
Thus the Wasserstein worst-case approach leads to a shift of the mean of the normal distribution in the direction 
$$-\nabla_a\sqrt{\int |\nabla_x f(x, a^\opt)|_s^2\, \mu(dx)}$$
compared to the non-robust case. In the simple case $\mu=\mathcal{N}(0,\sigma^2)$ with standard deviation $\sigma>0$ we obtain the MLE $\sigma^{\opt,N}=\frac{1}{N} \sum_{k=1}^N X_i^2$. We can directly compute (for $a=\sigma$) that
\begin{align*}
\nabla_\sigma \sqrt{\int \left|\nabla_x \left(\mathrm{const.}+\log\left(\exp\left(-\frac{x^2}{2(\sigma^\opt)^2}\right)\right)\right)\right|_s^2\, \mu(dx)}&=\nabla_\sigma \sqrt{\int \frac{x^2}{(\sigma^\opt)^4}\,\mu(dx)}\\
&=\nabla_\sigma\frac{\sigma^\opt}{(\sigma^\opt)^2}=\nabla_\sigma\frac{1}{\sigma^\opt}=-\frac{1}{(\sigma^\opt)^2}.
\end{align*}
Thus the robust approach accounts for a shift of $1/(\sigma^\opt)^2$ (of order 1 if mulitplied with inverse Fisher information) to account for a possibly higher variance in the underlying data.
In particular, in our approach, the so-called neutral spaces considered, e.g., in \cite[eq. (21)]{komorowski2011sensitivity} as
\begin{align*}
\left\{a:\ -(a-a^\opt)^T \nabla_a^2V(0,a^\opt) (a-a^\opt)\le \delta \right\}
\end{align*}
should also take this shift into account, i.e., their definition should be adjusted to
\begin{align*}
\Bigg\{a:\ -\left(a-a^\opt+\nabla_a\sqrt{\int |\nabla_x f(x, a^\opt)|_s^2\, \mu(dx)}\right)^T &\nabla_a^2V(0,a^\opt) \\
&\cdot \left(a-a^\opt+\nabla_a\sqrt{\int |\nabla_x f(x, a^\opt)|_s^2\, \mu(dx)}\right)\le \delta \Bigg\}.
\end{align*}
Lastly, let us mention another situation when our approach provides directly interpretable insights in the context of a parametric family of models. Namely, if one considers a family of models $\cP$ such that the worst-case model in the Wasserstein ball remains in $\cP$, i.e., $(x+f^\prime(x)\delta)_*\mu\in \mathcal{P}$, then considering (the first order approximation to) model uncertainty over Wasserstein balls actually reduces to considerations within the parametric family. While uncommon, such a situation would arise, e.g.,  for a scale-location family $\mathcal{P}$, with $\mu \in \mathcal{P}$ and a linear/quadratic $f$. 

\subsection{Link to the CLT of \cite{Blanchet:2019ue}}
\label{sec:CLT}
As observed in section \ref{sec:applications}\ref{sec:outofsample}above, Theorem \ref{thm:sens} allows to recover the main results in \cite{Blanchet:2019ue}. We explain this now in detail. 
Set $\|\cdot \|=|\cdot|_s$, $p=q=2$, $\mathcal{S}=\R^d$. Let $\mu_N$ denote the empirical measure of $N$ i.i.d.\ samples from $\mu$.
We impose the assumptions on $\mu$ and $f$ from \cite{Blanchet:2019ue}, including Lipschitz continuity of gradients of $f$ and strict convexity. These, in particular, imply that the optimizers $a^{\opt,N}_\delta, a^{\opt,N}$ and $a^{\opt}$, as defined in section \ref{sec:applications}\ref{sec:outofsample} are well defined and unique, 
and further $a^{\opt,N}_{1/\sqrt{N}}\to a^\opt$ as $N\to \infty$.
\cite[Thm.~1]{Blanchet:2019ue} implies that, as $N\to \infty$,
\begin{align}\label{eq:blanchet}
\sqrt{N} (a^{\opt,N}_{1/\sqrt{N}}-a^\opt) \Rightarrow (\nabla_a^2 V(0, a^\opt))^{-1} \cdot \left(H -\nabla_a\sqrt{\int |\nabla_x f(x, a^\opt)|_s^2\, \mu(dx)}\,\right),
\end{align}
where $H\sim \mathcal{N}\left(0, \int \nabla_a f(x,a^\opt)^T \nabla_a f(x, a^\opt)\, \mu(dx)\right).$ We note that for $\| \cdot \|=|\cdot|_s$ we have
\begin{align*}
h(x)=(\text{sign}(x_1)\,|x_1|^{s-1}, \dots, \text{sign}(x_k)\,|x_k|^{s-1})\cdot |x|_s^{1-s}=\nabla_x |x|_s.
\end{align*}
Thus
\begin{align*}
\nabla_a\sqrt{\int |\nabla_x f(x, a^\opt)|_s^2\, \mu(dx)} &= \frac{\int |\nabla_x f(x, a^\opt)|_s h(\nabla_x f(x, a^\opt)) \nabla_x \nabla_a f(x, a^\opt) \,\mu(dx)}{\sqrt{\int |\nabla_x f(x, a^\opt)|_s^2\, \mu(dx)}}\cdot 
\end{align*}
and \eqref{eq:blanchet} agrees with \eqref{eq:CLTfromOSE} which is justified by the Lipschitz growth assumptions on $f$, $\nabla_x f(x,a)$ and $\nabla_a \nabla_x f(x,a)$ from \cite{Blanchet:2019ue}, see \cite[eq.\ \eqref{eq:condit}]{SI}. 
In particular Theorem \ref{thm:sens} implies \eqref{eq:blanchet} as a special case. While this connection is insightful to establish\footnote{We thank Jose Blanchet for pointing out the possible link and encouraging us to explore it.} it is also worth stressing that the proofs in \cite{Blanchet:2019ue} pass through the dual formulation and are thus substantially different to ours. Furthermore, while Theorem \ref{thm:sens} holds under milder assumptions on $f$ than those in \cite{Blanchet:2019ue}, the last argument in our reasoning above requires the stronger assumptions on $f$. It is thus not clear if our results could help to significantly weaken the assumptions in the central limit theorems of \cite{Blanchet:2019ue}.

%\matmethods{
\section{Proofs} 
{
We consider the case $\mathcal{S}=\R^d$ and $\| \cdot\|=|\cdot|$ here. For the general case and additional details we refer to \cite{SI}. When clear from the context, we do not indicate the space over which we integrate. 
\begin{proof}[Proof of Theorem \ref{thm:main}]
	For every $\delta\geq 0$ let $C_{\delta}(\mu)$ denote those $\pi\in\mathcal{P}(\mathbb{R}^d\times \mathbb{R}^d)$ which satisfy 
	$$ \pi_1=\mu \text{ and }
    \left(\int |x-y|^p\,\pi(dx,dy)\right)^{1/p}\leq\delta.$$
    As the infimum in the definition of $W_p(\mu,\nu)$ is attained (see \cite[Theorem 4.1, p.43]{villani2008optimal}) one has $B_{\delta}(\mu)=\{ \pi_2 : \pi\in C_{\delta}(\mu)\}$.

    We start by showing the ``$\leq$'' inequality in the statement. 
    For any $a^\opt\in \Aoptim{0}$ one has $V(\delta)\leq\sup_{\nu\in B_\delta(\mu)} \int f(y,a^\opt)\,\nu(dy)$ with equality for $\delta=0$.
    Therefore, differentiating $f(\cdot,a^\opt)$ and using both Fubini's theorem and H\"older's inequality, we obtain that 
	\begin{align*}
	&V(\delta)-V(0)
	\leq \sup_{\pi\in C_\delta(\mu)}\int f(y,a^\opt)-f(x,a^\opt)\, \pi(dx,dy)\\
	&= \sup_{\pi\in C_\delta(\mu)} \int_0^1 \int \langle \nabla_x f(x+t(y-x),a^\opt),(y-x)\rangle\,\pi(dx,dy) dt\\
	&\leq \delta \sup_{\pi\in C_\delta(\mu)} \int_0^1 \Big(\int |\nabla_x f(x+t(y-x),a^\opt)|^q\pi(dx,dy)\Big)^{1/q} dt.
	\end{align*}
	Any choice $\pi^\delta\in C_\delta(\mu)$ converges in $p$-Wasserstein distance on $\cP(\mathbb{R}^d\times\mathbb{R}^d$) to the pushforward measure of $\mu$ under the mapping $x\mapsto (x,x)$, which we denote $[x\mapsto (x,x)]_\ast\mu$. This can be seen by, e.g., considering the coupling $[(x,y)\mapsto (x,y,x,x)]_\ast\pi^\delta$ between $\pi^\delta$ and $[x\mapsto (x,x)]_\ast\mu$.
	Now note that $q=p/(p-1)$ and the growth assumption on $\nabla_x f(\cdot, a^\opt)$ implies 
	\begin{align}
	\label{eq:bound.nabla.f.q.main}
	|\nabla_x f(x+t(y-x),a^\opt)|^q\leq c(1+|x|^p+|y|^p)
	\end{align}
	for some $c>0$ and all $x,y\in\mathbb{R}^d$, $t\in[0,1]$.
	In particular $\int |\nabla_x f(x+t(y-x),a^\opt)|^q\,\pi^\delta(dx,dy)\leq C$ for all $t\in[0,1]$ and small $\delta>0$, for another constant $C>0$.
	As further $(x,y)\mapsto |\nabla_x f(x+t(y-x),a^\opt)|^q$ is continuous for every $t$, the $p$-Wasserstein convergence of $\pi^\delta$ to $[x\mapsto (x,x)]_\ast\mu$ implies that
	\[\int |\nabla_x f(x+t(y-x),a^\opt)|^q\,\pi^{\delta}(dx,dy)
	\to\int |\nabla_x f(x,a^\opt)|^q\,\mu(dx)\]
	for every $t\in[0,1]$ for $\delta\to 0$, see \cite[Lemma \ref{lem:wasserstein.integrals.converge}]{SI}.
	Dominated convergence (in $t$) then yields ``$\leq$'' in the statement of the theorem. 

	We turn now to the opposite ``$\geq$'' inequality. 
    As $V(\delta)\geq V(0)$ for every $\delta>0$ there is no loss in generality in assuming that the right hand side is not equal to zero.
    Now take any, for notational simplicity not relabelled, subsequence of $(\delta)_{\delta>0}$ which attains the liminf in $(V(\delta)-V(0))/\delta$ and pick $a^\opt_\delta\in \Aoptim{\delta}$. By assumption, for a (again not relabelled) subsequence, one has $a^\opt_\delta\to a^\opt\in \Aoptim{0}$.
    Further note that $V(0)\leq\int f(x,a^\opt_\delta)\,\mu(dx)$ which implies 
    \begin{align*}
	V(\delta)-V(0)
	&\geq \sup_{\pi\in C_\delta(\mu)}\int f(y,a^\opt_\delta)-f(x,a^\opt_\delta) \,\pi(dx,dy).
	\end{align*}
	Now define $\pi^\delta:=[x\mapsto (x,x+\delta T(x))]_\ast\mu$, where
	\[T(x):= \frac{\nabla_x f(x,a^\opt)}{|\nabla_x f(x,a^\opt)|^{2-q}} \Big(\int |\nabla_x f(z,a^\opt)|^q\,\mu(dz)\Big)^{1/q-1}\]
	for $x\in\mathbb{R}^d$ with the convention $0/0=0$.
	Note that the integral is well defined since, as before in \eqref{eq:bound.nabla.f.q.main}, one has $|\nabla_x f(x,a^\opt)|^q\leq C(1+|x|^{p})$ for some $C>0$ and the latter is integrable under $\mu$.
	Using that $pq-p=q$ it further follows that
	\begin{align*}
	&\int |x-y|^p\,\pi^\delta(dx,dy)
	=\delta^p\int |T(x)|^p\,\mu(dx)\\
	&=\delta^p \frac{\int |\nabla_x f(x,a^\opt)|^{p+pq-2p} \,\mu(dx)}{\big(\int |\nabla_x f(z,a^\opt)|^q\,\mu(dz)\big)^{p(1-1/q)}} 
	= \delta^p.
	\end{align*}
	In particular $\pi^\delta\in C_\delta(\mu)$ and we can use it to estimate from below the supremum over $C_\delta(\mu)$ giving
	\begin{align*}
	\frac{V(\delta)-V(0)}{\delta}
	&\geq\frac{1}{\delta}\int f(x+\delta T(x),a^\opt_\delta)-f(x,a^\opt_\delta) \,\mu(dx)\\
	&=\int_0^1\int \langle \nabla_x f(x+t\delta T(x),a^\opt_\delta),T(x)\rangle \,\mu(dx)\,dt.
	\end{align*}
	For any $t\in [0,1]$, with $\delta\to 0$, the inner integral converges to 
	\begin{align*}
\int \langle \nabla_x f(x,a^\opt),T(x)\rangle \,\mu(dx) = \Big(\int |\nabla_x f(x,a^\opt)|^q \,\mu(dx)\Big)^{1/q}.
	\end{align*}
	The last equality follows from the definition of $T$ and a simple calculation. 
	To justify the convergence, first note that $\langle \nabla_x f(x+t\delta T(x),a^\opt_\delta),T(x)\rangle \to \langle \nabla_x f(x,a^\opt),T(x)\rangle $ for all $x\in\mathbb{R}^d$ by continuity of $ \nabla_x f$ and since $a^\opt_\delta \to a^\opt$. 
	Moreover, as before in \eqref{eq:bound.nabla.f.q.main}, one has $|T(x)|\leq c(1+|x|)$ for some $c>0$, hence $|\langle \nabla_x f(x+t\delta T(x),a^\opt), T(x)\rangle|\leq C(1+|x|^p)$ for some $C>0$ and all $t\in[0,1]$.
	The latter is integrable under $\mu$, hence convergence of the integrals follows from the dominated convergence theorem. This concludes the proof.
\end{proof}

\begin{proof}[Proof of Theorem \ref{thm:sens}]
We first show that
\begin{align}\label{eq:hedgesens.main}
\lim_{\delta \to 0}\frac{-\nabla_{a_i}V(0,a^\opt_\delta)}{\delta} &= \int \nabla_x \nabla_{a_i}f(x,a^\opt)\frac{\nabla_x f(x,a^\opt) }{ |\nabla_x f(x),a^\opt)|^{2-q}}\,\mu(dx)\\
&\qquad\cdot \Big( \int |\nabla_x f(x,a^\opt)|^q\,\mu(dx)\Big)^{1/q-1}\nonumber
\end{align}
for all $i \in \{1, \dots, k\}$.
We start with the ``$\le"$-inequality. For any $a\in\interior{\mathcal{A}}$ we have
\begin{align*}
\nabla_{a} f(y,a)-\nabla_{a} f(x,a)&=\int_0^1  \nabla_x \nabla_{a} f(x+t(y-x),a)(y-x)\rangle \,dt.
\end{align*}
Let $\delta > 0$ and recall that $a^\opt_\delta\in \Aoptim{\delta}$ converge to $a^\opt \in \Aoptim{0}$. 
Let $\Boptim_\delta(\mu,a^\opt_\delta)$ denote the set of $\nu\in B_{\delta}(\mu)$ which attain the value: $\int f(x,a^\opt_\delta)\,\nu(dx)=V(\delta)$. 
By \cite[Lemma \ref{lem:grad_hedge_new}]{SI} the function $a\mapsto V(\delta,a)$ is (one-sided) directionally differentiable at $a^\opt_\delta$ for all $\delta>0$ small and thus for all $i\in \{1, \dots, k\}$
\begin{align*}
 & \sup_{\nu \in \Boptim_\delta(\mu,a^\opt_\delta)} \int\nabla_{a_i} f(x,a^\opt_\delta)\,\nu(dx)\geq 0.
 \end{align*}
Then, using Lagrange multipliers to encode the optimality of $\Boptim_\delta(\mu,a^\opt_\delta)$ in $B_\delta(\mu)$, we obtain 
\begin{align*}
&-\nabla_{a_i}V(0,a^\opt_\delta)\le\sup_{\nu \in \Boptim_{\delta}(\mu,a^\opt_\delta)}\int \nabla_{a_i}f(y,a^\opt_\delta)\nu(dy)-\nabla_{a_i}V(0,a^\opt_\delta)\\
&=\sup_{\nu \in B_{\delta}(\mu) }\inf_{\lambda \in \R}\bigg(\int
\big[\nabla_{a_i}f(y,a^\opt_\delta)+\lambda (f(y,a^\opt_\delta)
-V(\delta))\big]\nu(dy)\\
&\quad - \int \big[\nabla_{a_i}f(x,a^\opt_\delta) +\lambda (f(x,a^\opt_\delta)-V(0,a^\opt_\delta))\big]\mu(dx)\bigg)\\
&=\inf_{\lambda\in \R} \bigg( \sup_{\pi \in C_{\delta}(\mu)} 
 \int_0^1  \int  \Big\langle \nabla_x \nabla_{a_i}f(x+t(y-x),a^\opt_\delta)\\
&\quad +\lambda \nabla_x f(x+t(y-x),a^\opt_\delta), y-x\Big\rangle \,\pi(dx,dy)\, dt\\
&\quad -\lambda \sup_{\pi \in C_{\delta}(\mu)} \int_0^1 \int  \langle \nabla_x f(x+t(y-x),a^\opt_\delta, y-x\rangle \, \pi(dx,dy)\,dt \bigg)
\end{align*}
where we used a minimax argument as well as Fubini's theorem. We note that the functions above satisfy the assumptions of Theorem \ref{thm:main} for a fixed $\lambda$. In particular using exactly the same arguments as in the proof of Theorem \ref{thm:main} (i.e., H\"older's inequality and a specific transport attaining the supremum) we obtain by exchanging the order of $\limsup$ and $\inf$ that
\begin{align}
&\limsup_{\delta \to 0}\frac{-\nabla_{a_i}V(0,a^\opt_\delta)}{\delta}\ \label{eq:gradVestimate}\\
&\leq \inf_{\lambda \in \R} \Bigg( \left( \int \left| \nabla_x \nabla_{a_i}f(x,a^\opt)+\lambda \nabla_x f(x,a^\opt) \right|^q\, \mu(dx) \right)^{1/q}\nonumber\\
&\qquad -\lambda \left( \int \left| \nabla_x f(x,a^\opt) \right|^q \, \mu(dx)\right)^{1/q}\Bigg).\nonumber
\end{align}
For $q=2$ the infimum can be computed explicitly and equals 
\begin{align*}
\frac{\int \langle \nabla_x \nabla_{a_i}f(x,a^\opt),\nabla_x f(x,a^\opt)\rangle\,\mu(dx)}{\sqrt{\int |\nabla_x f(x,a^\opt)|^2\,\mu(dx)}}
\end{align*} 
For the general case we refer to \cite[Lemma \ref{lem:inf}]{SI}, noting that by assumption $ \nabla_x f(x,a^\opt)\neq 0$, we see that the RHS above is equal to the RHS in \eqref{eq:hedgesens.main}. 

The proof of the ``$\ge"$-inequality in \eqref{eq:hedgesens.main} follows by the very same arguments.
Indeed, \cite[Lemma \ref{lem:grad_hedge_new}]{SI} implies that 
\[ \inf_{\nu \in \Boptim_\delta(\mu,a^\opt_\delta)} \int\nabla_{a_i} f(x,a^\opt_\delta)\,\nu(dx)\leq 0\]
for all $i\in\{1,\dots,k\}$  and we can write
\begin{align*}
&-\nabla_{a_i}V(0,a^\opt_\delta) 
\geq  \inf_{\nu \in \Boptim_\delta(\mu,a^\opt_\delta)} \int \nabla_{a_i} f(y, a^\opt_\delta)\,\nu(dy) -\nabla_{a_i}V(0,a^\opt_\delta)\\
&=\inf_{\nu \in B_{\delta}(\mu) }\sup_{\lambda \in \R}\bigg(\int
\big[\nabla_{a_i}f(y,a^\opt_\delta)+\lambda (f(y,a^\opt_\delta)
-V(\delta))\big]\nu(dy)\\
&\quad - \int \big[\nabla_{a_i}f(x,a^\opt_\delta) +\lambda (f(x,a^\opt_\delta)-V(0,a^\opt_\delta))\big]\mu(dx)\bigg).
\end{align*}
From here on, we argue as in the ``$\leq"$-inequality and conclude that indeed \eqref{eq:hedgesens.main} holds.

By assumption the matrix $\nabla_{a}^2 V(0,a^\opt)$ is invertible.
Therefore, in a small neighborhood of $a^\opt$, the mapping $\nabla_a V(0,\cdot)$ is invertible.
In particular $a^\opt_\delta=(\nabla_{a} V(0,\cdot))^{-1}\left(\nabla_{a}V(0,a^\opt_\delta)\right)$ and by the first order condition $a^\opt=(\nabla_{a} V(0,\cdot))^{-1}\left(0\right)$.
Applying the chain rule and using \eqref{eq:hedgesens.main} gives 
\begin{align*}
\lim_{\delta \to 0}\frac{a^\opt_\delta-a^\opt}{\delta}
&=(\nabla_{a}^2V(0,a^\opt))^{-1}\cdot \ \lim_{\delta \to 0}\frac{\nabla_{a}V(0,a^\opt_\delta)}{\delta}\\
&=- (\nabla^2_aV(0,a^\opt))^{-1}  \Big(\int |\nabla_x f(z,a^\opt)|^q\,\mu(dz)\Big)^{1/q-1} \\
&\cdot \int \frac{\nabla_{x}\nabla_af(x,a^\opt) \nabla_x f(x,a^\opt)}{|\nabla_x f(x,a^\opt)|^{2-q}} \, \mu(dx).
\end{align*}
This completes the proof.
\end{proof}

%\disclaimer{Insert disclaimer text here.}

%%%%%%%%%% Insert bibliography here %%%%%%%%%%%%%%

%Authors are thankful to St.\ John's College in Oxford for its support. 

% Bibliography
%\bibliographystyle{pnas-new}

%\bibliographystyle{RS} %%%% .BST file
\bibliographystyle{apalike}

\bibliography{bibsens,bibsens2}%%%%% .Bib file

\begin{appendix}

\section{Preliminaries}

We recall and further explain the setting from the main body of the paper \cite{main}.
Take $d,k\in \N$, endow $\R^d$ with the Euclidean norm $|\cdot|$. Throughout the paper we take the convention that topological properties, such as continuity or closure, are understood w.r.t.\ $|\cdot|$. We let $\interior{\Gamma}, \bar{\Gamma}, \partial\Gamma, \Gamma^c$ denote respectively the interior, the closure, the boundary and the complement of a set $\Gamma\subset\R^d$. We denote the set of all probability measures on $\Gamma$  by $\mathcal{P}(\Gamma)$. For a variable $\gamma\in \Gamma$, we will denote the optimizer by $\gamma^\opt$ and the set of optimizers by $\Gamma^\opt$. 

Fix a seminorm $\|\cdot\|$ on $\mathbb{R}^d$ and denote by $\|\cdot\|_\ast$ its (extended) dual norm, i.e.\ $\|y\|_\ast:=\sup\{ \langle x,y\rangle : \|x\|\leq 1\}$. 
Let us define the equivalence relation $x\sim y$ if and only if $\|x-y\|=0$. Furthermore let us set $U:=\{x\in \R^d \ :\ \|x\|=0\}$ and write $[x]=x+U$.  With this notation, the quotient space $\R^d/U=\{[x]\ :  \ x\in \R^d\}$ is a normed space for $\|\cdot\|$. Furthermore, by the triangle inequality for $\|\cdot\|$ and equivalence of norms on $\R^d$, there exists $c>0$ such that $\|x\| \leq c |x|$ and $|x|\le c \|x\|_\ast$ for all $x\in \R^d$. As $|\cdot|$ is Hausdorff, this immediately implies that $\|\cdot\|_\ast$ is Hausdorff as well. Furthermore we conclude, that $\|\cdot\|$ is continuous and $\|\cdot\|_\ast$ is lower semicontinuous w.r.t.\ $|\cdot|$ (as the supremum over continuous functions $\langle x, \cdot\rangle$). Lastly we make the convention that $B_{\delta}(x)$ denotes the ball of radius $\delta$ around $x$ in $|\cdot|$. As our setup is slightly non-standard, we state the following lemmas for completeness:

\begin{lemma}\label{dani:1}
	For every $x\in\mathbb{R}^d$ we have that $\|x\|=\sup \{\langle x, y\rangle \ : \ \|y\|_\ast \le 1\}$.
\end{lemma}

\begin{proof}
As $\{ x\in \mathcal{S}\ : \ \|x\|\le 1\}$ is convex and closed, this follows directly from the bipolar theorem.
\end{proof}

\begin{lemma}\label{dani:2}
Assume that $\|\cdot\|_\ast$ is strictly convex. Then the following hold:
\begin{enumerate}[(i)]
\item  For all $x\in \R^d$ there exists $h(x) \in \R^d$ such that $\|h(x)\|_\ast=1$ and $\|x\|=\langle x,h(x)\rangle$.
	If $x\neq 0$, then $h(x)$ is unique.
\item The map $h:\R^d\setminus \{0\} \to \R^d$ is continuous.
\end{enumerate}
\end{lemma}

\begin{proof}
Fix $x\in \R^d\setminus\{0\}$. The existence of $h(x)\in \R^d$ in \textit{(i)} follows from Lemma \ref{dani:1}. Assume towards a contradiction that there exists another $\tilde{h}(x)\in \R^d$ with $\|\tilde{h}(x)\|_\ast=1$, $\langle x,\tilde{h}(x)\rangle =\|x\|$ and $\tilde{h}(x)\neq h(x)$. Defining $\bar{h}(x)=(h(x)+\tilde{h}(x))/2$ we have $\langle x, \bar{h}(x)\rangle =(\langle x, h(x)\rangle+ \langle x, \tilde{h}(x) \rangle )/2=\|x\|$. On the other hand,  by the Hausdorff property of $\|\cdot\|_\ast$, we have $\|h(x) -\tilde{h}(x)\|_\ast \neq 0$ and thus, by strict convexity of $\|\cdot\|_\ast$, $\|\bar{h}(x)\|_\ast <1$. Using again Lemma \ref{dani:1}, we conclude $\|x\|\ge \langle x ,  \bar{h}(x)/\|\bar{h}(x)\|_\ast \rangle >\|x\|$, a contradiction.\\
For \textit{(ii)} we assume towards a contradiction that for some sequence $(x_n)_{n\in \N}$ in $\R^d$ we have $\lim_{n \to \infty }x_n=x\in \R^d\setminus\{0\}$, but $\lim_{n\to \infty} h(x_n)\neq h(x)$. As remarked above, we have $\{ \|\cdot\|_{\ast}\le 1 \} \subseteq B_c(0)$, in particular $\lim_{n \to \infty}h(x_n)=y\in \R^d$ after taking a subsequence. Recalling that $h(x)\neq y$ and $\|\cdot\|_\ast$ is lower semicontinuous, we conclude that $\|y\|_\ast \le 1$ and in particular $\|x\|> \langle x, y\rangle$ by Lemma \ref{dani:1} and \textit{(i)}. Finally $$\|x\|=\lim_{n \to \infty}\|x_n\|=\lim_{n \to \infty}\langle x_n, h(x_n) \rangle =\langle x, y\rangle,$$
which leads to a contradiction.
\end{proof}

\begin{lemma}\label{dani:3}
If $\|\cdot\|$ is strictly convex, then $\|\cdot\|_\ast$ is strictly convex as well.
\end{lemma}

\begin{proof}
Fix $y\in \R^d\setminus \{0\}$. We first note that 
\begin{align*}
k(y):=\{ x\in \R^d\ : \ \|x\|=1, \, \|y\|_\ast =\langle x, y\rangle\}/U
\end{align*}
is uniquely defined. Indeed, this follows from applying the exact same arguments as in the proof of Lemma \ref{dani:2}, adjusting for $U$.
Take now $y, y'\in \R^d$ such that $\|y\|_\ast=\|y'\|_\ast=1$ and $\|y-y'\|_\ast\neq 0$. Set $\bar{y}=(y+y')/2$ and note that $\|\bar{y}-y\|_\ast,\|\bar{y}-y'\|_\ast \neq 0$.  Then $\|\bar{y}\|_\ast=( \langle [k(\bar{y})] , y \rangle+\langle  [k(\bar{y})], y' \rangle )/2<1$. This shows the claim.
\end{proof}

Let $\mathcal{S}$ denote the state space which is a closed convex subset of $\R^d$. Fix $p>1$ and take $q=p/(p-1)$ so that $1/p+1/q=1$.
For probability measures $\mu$ and $\nu$ on $\mathcal{S}$, we define their $p$-Wasserstein distance as 
\begin{equation*}
    W_p(\mu, \nu):=\inf\left\{\int_{\mathcal{S}\times \mathcal{S}} \|x-y\|_\ast ^p\,\pi(dx,dy)\colon \pi \in \mathrm{Cpl}(\mu,\nu) \right\}^{1/p},
\end{equation*}
where $\mathrm{Cpl}(\mu,\nu)$ is the set of all probability measures $\pi\in\mathcal{P}(\mathcal{S}\times \mathcal{S})$ with first marginal $\pi_1:=\pi(\cdot\times\mathcal{S})=\mu$ and second marginal $\pi_2:=\pi(\mathcal{S}\times\cdot)=\nu$.
In the proofs we sometimes also use the $p$-Wasserstein distance with respect to the Euclidean norm $|\cdot|$ given by
$$W_{p}^{|\cdot|}(\mu, \nu)= \inf\left\{\int_{\mathcal{S}\times \mathcal{S}} |x-y|^p\,\pi(dx,dy)\colon \pi \in \mathrm{Cpl}(\mu,\nu) \right\}^{1/p}.$$
Recall that $|\cdot|\leq c\|\cdot\|_\ast$ for some constant $c>0$, which in turn implies that $W_{p}^{|\cdot|}(\cdot, \cdot)\leq c W_{p}(\cdot, \cdot)$.
A Wasserstein ball of size $\delta\geq 0$ around $\mu$ is denoted
\begin{equation*}
    B_{\delta}(\mu) := \left\{\nu \in \mathcal{P}(\mathbb{R}^d) :  W_p(\mu,\nu) \leq \delta  \right\}.
\end{equation*}
From now on, we fix $\mu\in \cP(\cS)$ such that $\mu(\partial\cS)=0$ and $\int_{\cS} |x|^p\,\mu(dx)<\infty$.
Let $\cA$ denote the action (decision) space which is a convex and closed subset of $\R^k$.  
We consider robust stochastic optimization problem \eqreff{eq. minmax}:
\begin{align*}
V(\delta)\coloneq\inf_{a\in\cA}V(\delta,a)\coloneq \inf_{a\in\cA} \sup_{\nu \in B_{\delta}\left( \mu \right)}\int_{\mathcal{S}} f\left( x,a \right)\,\nu(dx).
\end{align*}
In accordance with our conventions, we write $a^\opt$ for an optimizer: $V(\delta)=V(\delta,a^\opt)$ and $\Aoptim\subset \cA$ for the set of such optimizers. We also let $\Boptim_\delta(\mu,a)$ denote the set of measures $\nu^\opt$ such that $V(\delta,a)=\int_{ \mathcal{S}} f(x,a) \,\nu^\opt(dx)$ and sometimes write $\Boptim_\delta(\mu)$ for $\Boptim_\delta(\mu,a^\opt)$ if $a^\opt\in \Aoptim{\delta}$ is fixed.

\section{Discussion, extensions and proofs related to Theorem \ref{thm:main}}

We complement now the discussion of Theorem \ref{thm:main}. We start with some remarks, extensions and further examples before proceeding with the proofs, including a complete proof of Theorem \ref{thm:main} for general seminorms $\|\cdot\|$.

\subsection{Discussion and extensions of Theorem \ref{thm:main}}
\label{sec:literature}

\begin{remark}\label{rk:failforp1}
Theorem \ref{thm:main} may fail for $p=1$. Indeed take $d=1$, $\|\cdot\|=|\cdot |$ and $f(x)=x^2$, $\mathcal{S}=[-1,1]$, $\mu$ the point mass in zero, $\mu=\delta_0$. Then $\nabla_x f(x)=2x$ and the $\mu(dx)$--essential supremum of $|\nabla_x f(x)|$ is equal to $0$. However $\nu_\lambda:=\lambda\delta_1+(1-\lambda)\delta_0\in B_\lambda(\mu)$ for all $\lambda\in [0,1]$ and it is easy to see $V(\delta)=\delta$ and thus $V^\prime(0)=1$. The point where the proof of Theorem \ref{thm:main} breaks down is that the map $\delta$ to the $\nu_\delta$--essential supremum of $|\nabla_x f(x)|$ is not continuous at $\delta=0$.
\end{remark}

\begin{remark}
	Let $p>2$. In addition to Assumption \ref{ass:main}, suppose that $f$ is twice continuously differentiable and that for ever $r\geq 0$ there is $c\geq 0$ such that $|\nabla_x^2 f(x,a)|\leq  c(1+|x|^{p-2})$ for all $x\in\mathcal{S}$ and all $a\in\mathcal{A}$ with $|a|\leq r$.
	Then, the same arguments as in the proof of Theorem \ref{thm:main} but with a second order Taylor expansion yield
	\begin{align*}
	V(\delta)
	\leq V(0) & + \delta \left( \int_\mathcal{S} \| \nabla_x f(a^\opt,x) \|^q \, \mu(dx) \right)^{1/q} \\
	& + \delta^2 \left( \int_\mathcal{S} \lambda_{\max}\Big( \frac{1}{2} \nabla_x^2 f(a^\opt,x)\Big)^{r} \, \mu(dx) \right)^{1/r} + o(\delta^2),
	\end{align*}
	for small $\delta\geq0$, where $\lambda_{\max}$ denotes the largest eigenvalue of the Hessian taken w.r.t.\ the norm $\|\cdot\|_\ast$ and $r=p/(p-2)$ is such that $2/p + 1/r=1$.
	
	In particular, this means that if the term in front of $\delta^2$ is the same order of magnitude as the term in front of $\delta$, then the first order approximation is quite accurate for small $\delta$.
	Note that larger $p$ implies smaller $r$ and therefore a smaller term in front of the $\delta^2$ term.
\end{remark}

\begin{remark}\label{rk:abscont}
We believe that Assumption \ref{ass:main} lists natural sufficient conditions for differentiability of $V(\delta)$ in zero. In particular all these conditions are used in the proof of Theorem \ref{thm:main}. Relaxing Assumption \ref{ass:main} seems to require a careful analysis of the interplay between (the space explored by balls around) $\mu$ and the functions $f, \nabla_x f$. We state here a straightforward extension to the case where $f$ is only weakly differentiable and leave more fundamental extensions (e.g., to manifolds) for future research.\\
Specifically, in case that the baseline distribution $\mu$ is absolutely continuous w.r.t.\ the Lebesgue measure and $\| \cdot \|=|\cdot|$, Theorem \ref{thm:main} remains true if we merely assume that $f(\cdot,a)$ has a weak derivative (in the Sobolev sense) on $\interior{\cS}$ for all $a\in \mathcal{A}$ and replace $\nabla_x f(\cdot, a)$ by the weak derivative of $f(\cdot,a)$ in Assumption \ref{ass:main}. More concretely the first point of Assumption \ref{ass:main} should read:
\begin{itemize}
\item The weak derivative $(x,a)\mapsto  g(x,a)$ of $f(\cdot,a)$ is continuous at every point $(x,a)\in N\times  \mathcal{A}^\opt(0)$, where $N$ is a Lebesgue-null set, and for every $r>0$ there is $c>0$ such that $|g(x,a)|\leq c(1+|x|^{p-1})$ for all $x\in\cS$ and $|a|\leq r$. 
	\end{itemize}
\end{remark}
\begin{proof}[Proof of Remark \ref{rk:abscont}]
For notational simplicity we only consider the case $\mathcal{S}=\R^d$. 
Note that by, e.g., \cite{brezis2010functional}[Theorem 8.2] we can assume that $f(\cdot, a)$ is continuous and satisfies
\begin{align*}
f(y,a)-f(x,a)=\int_0^1 \langle g(x+t(y-x),a),y-x\rangle\, dt
\end{align*} 
for all $x,y\in \R^d$ and all $a\in \mathcal{A}$. 
Furthermore
\begin{align}\label{eq:absolute_continuity}
\sup_{\nu \in B_{\delta}(\mu)} \int_{\mathcal{S}} f(x,a)\,\nu(dx)
=\sup_{\nu \in B_{\delta}(\mu), \ \nu \ll \mathrm{Leb}} \int_{\mathcal{S}} f(x,a)\,\nu(dx),
\end{align}
where $\nu\ll\mathrm{Leb}$ means that $\nu$ is absolutely continuous w.r.t.\ the Lebesgue measure. Indeed, let us take $\nu  \in B_{\delta}(\mu)$ and set $\tilde{\nu}=\tilde{\nu}(t, \epsilon)=(1-t)(\nu \ast N(0,\epsilon))+t\mu$, where $N(0, \epsilon)$ denotes the multivariate normal distribution 
with covariance $\epsilon \mathbf{I}$, $\epsilon>0$ and $\ast$ denotes the convolution operator. 
For every $0<t<1$, by convexity of $W_p^p(\cdot, \cdot)$ and the triangle inequality for $W_p$, we have 
\begin{align*}
W_p^p(\mu,\tilde{\nu})
&\le (1-t)W_p^p(\nu \ast N(0,\epsilon),\mu)+tW_p^p(\mu,\mu)\\
&= (1-t)W_p^p(\nu \ast N(0,\epsilon),\mu)\\
&\leq (1-t) \left( W_p(\nu \ast N(0,\epsilon),\nu)+ W_p(\nu,\mu) \right)^p.
\end{align*}
By assumption $W_p(\nu,\mu)\leq\delta$ and one can check that $W_p(\nu \ast N(0,\epsilon),\nu)\to 0$ as $\varepsilon\to0$.
Hence, for every $t<1$ there exists small $\varepsilon=\varepsilon(t)>0$ such that $W_p(\mu,\tilde{\nu})\leq \delta$.
As further $\lim_{t \to 1} \int_{ \mathcal{S}} f(x,a) \,\tilde{\nu}(dx)=\int_{ \mathcal{S}} f(x,a)\, \nu(dx)$, this shows \eqref{eq:absolute_continuity}.  
The proof of the remark now follows by the exact same arguments as in the proof of Theorem \ref{thm:main}.
\end{proof}

A natural example, which highlights the importance of  Remark \ref{rk:abscont} is the following:

\begin{example}
We let $\mu$ be a model for a vector of returns $X\in \mathcal{S}=\R^d$ and assume that $\mu$ is absolutely continuous with respect to Lebesgue measure. Let further $\| \cdot \|=|\cdot|$ and let $z\in B\subset \R^d$ denote a portfolio. We then consider the average value at risk  at level $\alpha\in(0,1)$ of the portfolio wealth $\langle z, X \rangle$, which can be written as
\begin{align*}
\text{AV@R}_{\alpha}\left(\langle z, X \rangle \right)=\frac{1}{\alpha} \int_{1-\alpha}^1 \text{V@R}_{u}(\langle z, X \rangle) du,
\end{align*}
where $\text{V@R}_u(\langle z, X \rangle)$ is the value at risk at level $u\in (0,1)$ defined as
\begin{align*}
\text{V@R}_u(\langle z, X \rangle)=\inf\{ x\in \R^d\  : \mu(\langle z, x\rangle)\ge u\}.
\end{align*}
We note that the average value at risk is an example for an optimized certainty equivalent (OCE), when choosing $l(x,a)=a+  \frac{1}{\alpha}  (x-a)^+$ in \cite[p.~3]{main}. We can thus rewrite the optimization problem
\begin{align*}
                    V(0) & = \inf_{z\in B} \text{AV@R}_{\alpha}\left(\langle z, X \rangle \right)
\end{align*}       
as 
\begin{align*}           
                 V(0)   = \inf_{z\in B, m\in \R}  \left(  m+\frac{1}{\alpha}\int_{\mathcal{S}} \left( \langle z ,x\rangle -m \right)^+ \mu(dx) \right) .
\end{align*}
Set $\mathcal{A}=B\times \R$ and assume that there exists a unique minimiser $(z^\opt, m^\opt) \in \Aoptim{0}$ of $V(0)$. Then $m^\opt$ is given by $ \mathrm{V@R} (\langle z^\opt, X \rangle)$.
The robust version of $V(0)$ reads
                \begin{equation*}
                  V(\delta) =
                  \inf_{(z,m) \in \mathcal{A}} \sup_{\nu \in B_\delta(\mu)}  \left( m+\frac{1}{\alpha}\int_{\mathcal{S}} \left( \langle z,x \rangle - m \right)^+ \nu(dx)\right). 
               \end{equation*}
Note that the function $x\mapsto x^+$ is weakly differentiable with weak derivative $\mathbf{1}_{\{x\ge 0\}}$. In conclusion $f(x,(z, m))=m+\frac{1}{\alpha} \left( \langle z, x \rangle - m \right)^+$ has weak derivative $$g(x,(z,m))=\frac{1}{\alpha} \mathbf{1}_{\{\langle z, x \rangle - m\ge 0\}},$$
which is continuous at $(x, (h^\opt, m^\opt))$ except on the lower-dimensional set $\{x\in\mathcal{S}\ : \ \langle z^\opt, x \rangle - m^\opt=0\}$, which is in particular a Lebesgue null set. Remark \ref{rk:abscont} thus yields
        \begin{align*}
            &V'(0) = |z^\opt|\left(\frac{1}{\alpha^q}\int_{\mathcal{S}} \mathbf{1}_{\left\{ \langle z^\opt  x\rangle \geq  \mathrm{V@R}_{\alpha}  \left( \langle z^\opt ,X\rangle  \right) \right\}}\mu(dx)\right)^{\frac{1}{q}} = \frac{| z^\opt |}{\alpha^{1/p}}
            \end{align*}
            and thus 
            \begin{align*}
            V(\delta) &= \text{AV@R}_{\alpha}\left( \langle z^\opt , X\rangle \right) + \frac{ |z^\opt |}{\alpha^{1/p}}\delta + o(\delta).
        \end{align*}
  Comparing with \cite[Table 1]{Bartl:2019hk}, we see that this approximation is actually exact for $p=1,2$.
\end{example}

We now mention two extensions of Theorem \ref{thm:main}. The first one concerns the derivative of $V(\delta)$ for $\delta>0$. 

\begin{corollary}
\label{Cor:sensitivity}
Fix $r>0$ and in addition to the assumptions of Theorem \ref{thm:main}  assume that 
\begin{itemize}
\item  $\Aoptim{r+\delta}\neq\emptyset$ for $\delta\geq 0 $ small enough and for every sequence $(\delta_n)_{n\in \N}$ such that $\lim_{n\to \infty} \delta_n=0$ and $(a^\opt_n)_{n\in \N}$ such that $a^\opt_n\in \Aoptim{r+\delta_n}$ there is a subsequence which converges to some $a^\opt \in \Aoptim{r}$.
\item there exists $\varepsilon>0$ such that for all $\gamma>0$ and every $a\in\mathcal{A}$ with $|a|\le \gamma$ one has $|\nabla_x f(x,a)|\leq c(1+|x|^{p-1-\varepsilon})$ for all $x\in\mathcal{S}$ and some constant $c>0$.
\end{itemize}
Then
    \[ V'(r+)=\lim_{\delta\to 0}\frac{V(r+\delta)-V(r)}{\delta}= \inf_{a^\opt\in \Aoptim{r}} 
    \sup_{\nu \in \Boptim_r(\mu,a^\opt)}\left( \int_{\mathcal{S}} \|\nabla_x f(x,a^\opt)\|^q\,\nu(dx)\right)^{1/q},\]
where we recall that $\Boptim_r(\mu,a^\opt)$ is the set of all $\nu\in B_r(\mu)$ for which $\int_{\mathcal{S}} f(x,a^\opt)\,\nu(dx)=V(r)$.
\end{corollary}
\begin{remark}
 Recall the notation $V(\delta,a)$ in \eqref{eq. minmax}.  
 Inspecting the proof of the above Corollary, it is clear the main difficulty is in showing that 
\[ \lim_{\delta\to 0}\frac{V(r+\delta,a)-V(r,a)}{\delta}=
    \sup_{\nu \in \Boptim_r(\mu,a)}\left( \int_{\mathcal{S}} \|\nabla_x f(x,a)\|^q\,\nu(dx)\right)^{1/q}.\]
In this way, the final statement of Corollary \ref{Cor:sensitivity}, or indeed of Theorem \ref{thm:main}, can be interpreted as an instance of the envelope theorem.
\end{remark}
 The second extension of Theorem \ref{thm:main} offers a more specific sensitivity result by including additional constraints on the ball $B_{\delta}(\mu)$ of measures considered. Let $m\in\mathbb{N}$ and let $\Phi=(\Phi_1,\dots,\Phi_m)\colon\mathcal{S}\to\mathbb{R}^m$ be a family of $m$ functions and assume that $\mu$ is calibrated to $\Phi$ in the sense that $\int_{\mathcal{S}} \Phi(x)\,\mu(dx)=0$.
Consider the set 
\[B^\Phi_\delta(\mu):=\left\{ \nu\in B_\delta(\mu) : \int_{\mathcal{S}} \Phi(x)\,\nu(dx)=0  \right\} \]
and the corresponding optimization problem
\begin{align*}
	V^{\Phi}(\delta)\coloneq\inf_{a\in\mathcal{A}}\sup_{\nu \in B^\Phi_{\delta}\left( \mu \right)}\int_{\mathcal{S}} f\left( x,a \right)\,\nu(dx).
\end{align*}
We have the following result.

\begin{theorem}[Sensitivity of $V(\delta)$ under linear constraints]
\label{thm.constraints}
	In addition to the assumptions of Theorem \ref{thm:main}, assume that there is some small $\varepsilon>0$ such that for every $a\in\mathcal{A}$ one has $|f(x,a)|\leq c(1+|x|^{p-\varepsilon})$ for all $x\in\mathbb{R}^d$ and some constant $c>0$.
	Further assume that $\Phi_i$, $i\leq m$,  are continuously differentiable with $|\Phi_i(x)|\leq c(1+|x|^{p-\varepsilon})$, $|\nabla_x \Phi_i(x)|\leq c(1+|x|^{p-1})$ and that the  non-degeneracy condition
	\begin{align}\label{eq:non-degenerate}
	\inf\left\{ \int_{\mathcal{S}} \bigg\| \sum_{i=1}^m \lambda_i \nabla_x\Phi_i(x) \bigg\|^q\,\mu(dx) : \lambda\in \R^d,\ |\lambda|=1\right\}>0
	\end{align}
	holds. Then
	\begin{align*}
	(V^{\Phi})'(0)= \inf_{a^\opt\in \Aoptim{0}}\inf_{\lambda\in\mathbb{R}^m}\left(\int_{\mathcal{S}} \Big\|\nabla_x f(x,a^\opt)+\sum_{i=1}^m \lambda_i\nabla_x\Phi_i(x)\Big\|^q\,\mu(dx) \right)^{1/q}.
	\end{align*}
\end{theorem}

\begin{remark}\label{redudant}
	Note that if $\|\cdot\|$ is a norm and $\mu$ has full support, the above non-degeneracy condition \eqref{eq:non-degenerate} can be made without loss of generality. 
	Indeed, as the unit circle is compact and the function $\lambda \mapsto \int_\mathcal{S} \left\| \sum_{i=1}^m \lambda_i \nabla_x\Phi_i(x) \right\|^q\,\mu(dx)$ is continuous, the infimum in \eqref{eq:non-degenerate} is attained. 
	In particular, if 
	\begin{align*}
	\inf\left\{ \int_{\mathcal{S}} \left\| \sum_{i=1}^m \lambda_i \nabla_x\Phi_i(x) \right\|^q\,\mu(dx) : |\lambda|=1\right\}=0,
	\end{align*}
	then $\sum_{i=1}^m \lambda_i \nabla_x\Phi_i=0$ $\mu$-a.s.\ for some $\lambda$ in the unit circle. As $\mu$ has full support this implies that $\sum_{i=1}^m \lambda_i \nabla_x\Phi_i=0$ on $\mathcal{S}$. Thus $\nabla_x\Phi_1, \dots, \nabla_x\Phi_m$ are  linearly dependent functions on $\mathcal{S}$.
	Deleting all linearly dependent coordinates and calling the resulting vector $\tilde{\Phi}$, we have $V^\Phi(\delta)=V^{\tilde{\Phi}}(\delta)$ for every $\delta\geq 0$.
	Moreover, the non-degeneracy condition \eqref{eq:non-degenerate} holds for $\tilde{\Phi}$.
\end{remark}

\begin{remark}
We can relax the conditions of Theorem \ref{thm.constraints} in the spirit of Remark \ref{rk:abscont}: more specifically, assume that the baseline distribution $\mu$ is absolutely continuous w.r.t.\ the Lebesgue measure and $\| \cdot \|=|\cdot|$. Then Theorem \ref{thm.constraints} remains true if we merely assume that $f(\cdot,a)$ and $\Phi_i$ have a weak derivative (in the Sobolev sense) on $\interior{\cS}$ for all $a\in \mathcal{A}$ and replace $\nabla_x f(\cdot, a)$ and $\nabla \Phi_i$ by the weak derivative of $f(\cdot,a)$ and of $\Phi_i$ respectively. More concretely the assumption should read:
\begin{itemize}
\item The weak derivatives $(x,a)\mapsto  g(x,a)$ of $f(\cdot,a)$ and $x\mapsto g_i(x)$ of $\Phi_i$ are continuous at every point $(x,a)\in N\times \mathcal{A}^{\opt}(0)$, where $N$ is a Lebesgue-null set, and for every $r>0$ there is $c>0$ such that $|g_i(x,a)|\leq c(1+|x|^{p-1})$ and $|g_i(x)|\leq c(1+|x|^{p-1})$ for all $x\in\cS$, $i=1, \dots,m$  and $|a|\leq r$. 
	\end{itemize}
\end{remark}

\begin{example}[Martingale constraints]
\label{ex:martingale} Let $d=1$, $\mathcal{S}=\R$, $\|\cdot\|=|\cdot|$, $p=2$, and let $\Phi_1(x):=x-x_0$ and $\Phi:=\{\Phi_1\}$, i.e., $B^\Phi_\delta(\mu)$ corresponds to the measures $\nu \in B_{\delta}(\mu)$ satisfying the martingale (barycentre preservation) constraint $\int_{ \mathbb{R} } x\,\nu(dx)=x_0$. Clearly the assumptions on $\Phi$ of Theorem \ref{thm.constraints} are satisfied. It remains to solve the  optimization problem over $\lambda\in\mathbb{R}$ and plug in the optimizer. We then obtain
    \[ (V^{\Phi})'(0)
    =\inf_{a^\opt \in \Aoptim{0}} \left(\int_{\R} \left(\nabla_x f(x,a^\opt)-\int_{\R} \nabla_x f(y,a^\opt)\,\mu(dy)\right)^2\,\mu(dx) \right)^{1/2}, \]
	i.e., $ (V^{\Phi})'(0)$ is the standard deviation of $\nabla_x f(\cdot,a^\opt)$ under $\mu$. In line with the previous remark, this results extend to the case of the call option pricing discussed in the main body of the paper.	
\end{example}

\begin{example}[Covariance constraints]
\label{ex:covariance} Let $d=2$, $\mathcal{S}=\R^2$,  $\|\cdot\|=|\cdot|$, $p=2$.
   Further let $\Phi_1(x_1,x_2):=x_1 x_2-b$ for some $b\in \R$ and $\Phi:=\{\Phi_1\}$, i.e., we want to optimize over measures $\nu \in B_{\delta}(\mu)$ satisfying the covariance constraint $\int_{ \mathbb{R}^2} x_1x_2\,\nu(dx)=b$.  Assume that there exists no $\lambda\in \R\setminus\{0\}$ such that  $\mu$-a.s. $x_1=\lambda x_2$.
Clearly the assumptions on $\Phi$ of Theorem \ref{thm.constraints} are satisfied. 
Note that 
\begin{align*}
&\int_{\R^2} |\nabla_x f(x,a)+\lambda_1 \nabla_x \Phi_1(x)|^2\,\mu(dx)  \\
&=\int_{\R^2} (\nabla_{x_1} f(x,a)+\lambda_1x_2)^2+(\nabla_{x_2} f(x,a)+\lambda_1 x_1)^2\,\mu(dx),
\end{align*}
so in particular	 the optimal $\lambda$ in the definition of $(V^\Phi)'(0)$ is given by
$$\lambda_1=\frac{-\int_{\R^2} \nabla_{x_1}f(x,a)x_2+\nabla_{x_2}f(x,a)x_1\,\mu(dx)}{\int_{\R^2} x_1^2+x_2^2\,\mu(dx)}.$$
Plugging this in gives
\begin{align*}
&\int_{\R^2} |\nabla_x f(x,a)+\lambda_1 \nabla_x \Phi_1(x)|^2\,\mu(dx) =\int_{\R^2} (\nabla_{x_1}f(x,a))^2+(\nabla_{x_2}f(x,a))^2\,\mu(dx)\\
&\qquad\qquad+2\lambda_1 \int_{\R^2}( \nabla_{x_1}f(x,a)x_2+\nabla f_{x_2}f(x,a) x_1 )\mu(dx) +\lambda_1^2  \int_{\R^2} x_2^2+x_1^2\, \mu(dx)\\
& =\int_{\R^2} (\nabla_{x_1}f(x,a))^2+(\nabla_{x_2}f(x,a))^2\,\mu(dx)-\frac{\left(\int_{\R^2} (\nabla_{x_1}f(x,a) x_2+\nabla_{x_2}f(x,a)x_1)\,\mu(dx)\right)^2}{\int_{\R^2} (x_1^2+x_2^2)\,\mu(dx)}.
\end{align*}
It follows that 
   \begin{align*}
    (V^{\Phi})'(0)
    =\inf_{a^\opt \in \Aoptim{0}} \Bigg( &\int_{\R^2} |\nabla_{x}f(x,a^\opt)|^2\,\mu(dx)\\
    &-\frac{\big(\int_{\R^2}     \nabla_{x_1}f(x,a^\opt) x_2 +\nabla_{x_2}f(x,a^\opt)x_1
 \,\mu(dx)\big)^2}{\int_{\R^2} |x|^2 \,\mu(dx)}\Bigg)^{1/2}.   
 \end{align*}
\end{example}

\begin{example}[Calibration] 
Consider the function $f((T, K), a)= (E_{\P_a}[(S_T-K)^+]-C((T,K))^2$, the discrete measure $\mu$ formalises grid points for which option data $C(T,K)$ is available, $\cS\subset \R_+\times \R_+$ is the set of maturities and strikes of interest and $\{\P_a, a\in \cA\}$, for a given compact set $\cA$, is a class of parametric models (e.g., Heston). A Wasserstein ball around $\mu$ can then be seen as a plausible formalisation of market data uncertainty. 
Derivatives in $T$ and $K$ correspond to classical pricing sensitivities, which are readily available for most common parametric models. These have to be only evaluated for one model $\P_{a^\opt}.$ Changing the class of parametric models $\{\P_a, a\in A\}$ and computing the sensitivity in Theorem \ref{thm:main} could then yield insights into when a calibration procedure can be considered reasonably robust.  
\end{example}

\subsection{Proofs and auxiliary results related to Theorem \ref{thm:main}}

\begin{proof}[Proof of Theorem \ref{thm:main}] We present now a complete proof of Theorem \ref{thm:main} for general state space $\cS$ and semi-norm $\|\cdot \|$. All the essential ideas have already been outlined in \cite{main} but for the convenience of the reader we repeat all of the steps as opposed to only detailing where the general case differs from the one treated in \cite{main}.

\textbf{Step 1:} Let us first assume that $\mathcal{S}=\R^d$.
For every $\delta\geq 0$ let $C_{\delta}(\mu)$ denote those $\pi\in\mathcal{P}(\mathcal{S}\times \mathcal{S})$ which satisfy 
	$$ \pi_1=\mu \text{ and }
    \left(\int_{\mathcal{S}\times \mathcal{S}} \|x-y\|^p_\ast\,\pi(dx,dy)\right)^{1/p}\leq\delta.$$
    Note that the dual norm $\|\cdot\|_\ast$ is lower semicontinuous, which implies that the infimum in the definition of $W_p(\mu,\nu)$ is attained (see \cite[Theorem 4.1, p.43]{villani2008optimal}) one has $B_{\delta}(\mu)=\{ \pi_2 : \pi\in C_{\delta}(\mu)\}$.

    We start by showing the ``$\leq$'' inequality in the statement. 
    For any $a^\opt\in \Aoptim{0}$ one has $V(\delta)\leq\sup_{\nu\in B_\delta(\mu)} \int_{ \mathcal{S}} f(y,a^\opt)\,\nu(dy)$ with equality for $\delta=0$.
    Therefore, differentiating $f(\cdot,a^\opt)$ and using Fubini's theorem, we obtain that 
	\begin{align*}
	V(\delta)-V(0)
	&\leq \sup_{\pi\in C_\delta(\mu)}\int_{\mathcal{S}\times \mathcal{S}} f(y,a^\opt)-f(x,a^\opt)\, \pi(dx,dy)\\
	&= \sup_{\pi\in C_\delta(\mu)} \int_0^1 \int_{\mathcal{S}} \langle \nabla_x f(x+t(y-x),a^\opt),(y-x)\rangle\,\pi(dx,dy) dt.
	\end{align*}
	Now recall that $\langle x,y\rangle \leq \|x\| \|y\|_\ast$ for every $x,y\in\mathbb{R}^d$, whence for any $\pi\in C_\delta(\mu)$ and $t\in[0,1]$, we have that
	\begin{align*}
	&\int_{\mathcal{S}} \langle \nabla_x f(x+t(y-x),a^\opt),(y-x)\rangle \, \pi(dx,dy) \\
	&\leq \int_{\mathcal{S}} \| \nabla_x f(x+t(y-x),a^\opt)\|   \|y-x\|_\ast \,\pi(dx,dy) \\
	&\leq \Big( \int_{\mathcal{S}} \| \nabla_x f(x+t(y-x),a^\opt)\|^q  \, \pi(dx,dy) \Big)^{1/q} \Big( \int_{\mathcal{S}} \| y-x \|^p\, \pi(dx,dy) \Big)^{1/p},
	\end{align*}
	where we used H\"older's inequality to obtain the last inequality. 
	By definition of $C_\delta(\mu)$ the last integral is smaller than $\delta$ and we end up with 
	\[ V(\delta)-V(0)
	\leq \delta \sup_{\pi\in C_\delta(\mu)} \int_0^1 \Big(\int_{\mathcal{S}} \|\nabla_x f(x+t(y-x),a^\opt)\|^q\pi(dx,dy)\Big)^{1/q} dt.\]	
	It remains to show that the last term converges to the integral under $\mu$.
	To that end, note that any choice $\pi^\delta\in C_\delta(\mu)$ converges in $W_p^{|\cdot|}$ on $\cP(\mathcal{S}\times \mathcal{S}$) to the pushforward measure of $\mu$ under the mapping $x\to (x,x)$, which we denote $[x\mapsto (x,x)]_\ast\mu$. This can be seen by, e.g., considering the coupling $[(x,y)\mapsto (x,y,x,x)]_\ast\pi^\delta$ between $\pi^\delta$ and $[x\mapsto (x,x)]_\ast\mu$.
	Now note that, together with growth restriction on $\nabla_x f$ of Assumption \ref{ass:main}, $q=p/(p-1)$ implies 
	\begin{align}
	\label{eq:bound.nabla.f.q}
	\|\nabla_x f(x+t(y-x),a^\opt)\|^q\leq c(1+|x|^p+|y|^p)
	\end{align}
	for some $c>0$ and all $x,y\in\mathbb{R}^d$, $t\in[0,1]$.
	Recall that there furthermore exists $\tilde{c}>0$ such that $\|x\| \leq \tilde{c} |x|$, in particular $\int_\mathcal{S} \|\nabla_x f(x+t(y-x),a^\opt)\|^q\,\pi^\delta(dx,dy)\leq C$ for all $t\in[0,1]$ and small $\delta>0$, for another constant $C>0$.
	As Assumption \ref{ass:main} further yields continuity of $(x,y)\mapsto \|\nabla_x f(x+t(y-x),a^\opt)\|^q$ for every $t$, the $p$-Wasserstein convergence of $\pi^\delta$ to $[x\mapsto (x,x)]_\ast\mu$ implies that
	\[\int_{\mathcal{S}} \|\nabla_x f(x+t(y-x),a^\opt)\|^q\,\pi(dx,dy)
	\to\int_{\mathcal{S}} \|\nabla_x f(x,a^\opt)\|^q\,\mu(dx)\]
	for every $t\in[0,1]$, see Lemma \ref{lem:wasserstein.integrals.converge}.
	Dominated convergence (in $t$) then yields ``$\leq$'' in the statement of the theorem. 

	We turn now to the opposite ``$\geq$'' inequality. 
    As $V(\delta)\geq V(0)$ for every $\delta>0$ there is no loss in generality in assuming that the right hand side is not equal to zero.
    Now take any, for notational simplicity not relabelled, subsequence of $(\delta)_{\delta>0}$ which attains the liminf in $(V(\delta)-V(0))/\delta$ and pick $a^\opt_\delta\in \Aoptim{\delta}$. By the second part of Assumption \ref{ass:main}, for a (again not relabelled) subsequence, one has $a^\opt_\delta\to a^\opt\in \Aoptim{0}$.
    Further note that $V(0)\leq\int_\mathcal{S} f(x,a^\opt_\delta)\,\mu(dx)$ which implies 
    \begin{align*}
	V(\delta)-V(0)
	&\geq \sup_{\pi\in C_\delta(\mu)}\int_{\mathcal{S}\times \mathcal{S}} f(y,a^\opt_\delta)-f(x,a^\opt_\delta) \,\pi(dx,dy).
	\end{align*}
	By Lemma \ref{dani:2} there exists a function $h\colon \mathbb{R}^d\mapsto \{ x\in \R^d : \|x\|_\ast=1\}$ such that $\|x\|=\langle x, h(x) \rangle$ for every $x\in\mathbb{R}^d$. 
	Now define 
	\begin{align*}
	\pi^\delta&:=[x\mapsto (x,x+\delta T(x))]_\ast\mu, \quad\text{where} \\
	T(x)&:= \frac{h(\nabla_x f(x,a^\opt))}{\|\nabla_x f(x,a^\opt)\|^{1-q}} \Big(\int_{\mathcal{S}} \|\nabla_x f(z,a^\opt)\|^q\,\mu(dz)\Big)^{1/q-1}
	\end{align*}
	for $x\in\mathbb{R}^d$ with the convention $h(\cdot)/0=0$.
	Note that the integral is well defined since, as before in \eqref{eq:bound.nabla.f.q}, one has $\|\nabla_x f(x,a^\opt)\|^q\leq C(1+|x|^{p})$ for some $C>0$ and the latter is integrable under $\mu$.
	Using that $pq-p=q$ it further follows that
	\begin{align*}
	&\int_{\mathcal{S}\times \mathcal{S}} \|x-y\|_\ast^p\,\pi^\delta(dx,dy)
	=\delta^p\int_{\mathcal{S}} \|T(x)\|_\ast^p\,\mu(dx)\\
	&=\delta^p \frac{\int_{\mathcal{S}} \|\nabla_x f(x,a^\opt)\|^{pq-p} \,\mu(dx)}{\big(\int_{\mathcal{S}} \|\nabla_x f(z,a^\opt)\|^q\,\mu(dz)\big)^{p(1-1/q)}} 
	= \delta^p.
	\end{align*}
	In particular $\pi^\delta\in C_\delta(\mu)$ and we can use it to estimate from below the supremum over $C_\delta(\mu)$ giving
	\begin{align*}
	\frac{V(\delta)-V(0)}{\delta}
	&\geq\frac{1}{\delta}\int_{\mathcal{S}} f(x+\delta T(x),a^\opt_\delta)-f(x,a^\opt_\delta) \,\mu(dx)\\
	&=\int_0^1\int_{\mathcal{S}} \langle \nabla_x f(x+t\delta T(x),a^\opt_\delta),T(x)\rangle \,\mu(dx)\,dt.
	\end{align*}
	For any $t\in [0,1]$, with $\delta\to 0$, the inner integral converges to 
	\begin{align*}
\int_{\mathcal{S}} \langle \nabla_x f(x,a^\opt),T(x)\rangle \,\mu(dx) = \Big(\int_{\mathcal{S}} \|\nabla_x f(x,a^\opt)\|^q \,\mu(dx)\Big)^{1/q}.
	\end{align*}
	The last equality follows from the definition of $T$ and a simple calculation. 
	To justify the convergence, first note that 
	\[\langle \nabla_x f(x+t\delta T(x),a^\opt_\delta),T(x)\rangle \to \langle \nabla_x f(x,a^\opt),T(x)\rangle \]
	for all $x\in\mathbb{R}^d$ by continuity of $(a,x)\mapsto \nabla_x f(x,a)$ and since $a^\opt_{\delta} \to a^\opt$. 
	Moreover, as before in \eqref{eq:bound.nabla.f.q}, one has 
	\[ \|\langle \nabla_x f(x+t\delta T(x),a^\opt),T(x)\rangle \|\leq C(1+|x|^p)\]
	for some $C>0$ and all $t\in[0,1]$.
	The latter is integrable under $\mu$, hence convergence of the integrals follows from the dominated convergence theorem.\\
	
	\textbf{Step 2:} We now extend the proof to the case, where $\cS\subset \R^d$ is closed convex and its boundary has zero measure under $\mu$.\\
	Note that the proof of the ``$\le$"-inequality remains unchanged. We modify the proof of the ``$\ge$"-inequality as follows: let us first define 
	\[\mathcal{S}^\epsilon:=\{x\in \mathcal{S} \ :\ |x- z| \geq \epsilon \text{ for all }  z\in \mathcal{S}^c\}\]
	for all $\epsilon>0$, so that in particular $\bigcup_{\epsilon>0} \mathcal{S}^{\epsilon}=\interior{\cS}$.
	We now redefine
	$$\pi^\delta:=\left[x\mapsto \left(x,x+\delta T(x) \mathbf{1}_{\{x\in \mathcal{S}^{\sqrt{\delta}}\}} \mathbf{1}_{\{|T(x)|\le 1/\sqrt{\delta} \}}\right)\right]_\ast\mu.$$
	Then $\pi^{\delta}\in \mathcal{P}(\mathcal{S}\times \mathcal{S})$ and in particular $\pi^{\delta}\in C_{\delta}(\mu)$ as in Step 1. Noting that 
	\begin{align*}
	 \lim_{\delta \to 0} T(x) \mathbf{1}_{\{x\in \mathcal{S}^{\sqrt{\delta}}\}} \mathbf{1}_{\{|T(x)|\le 1/\sqrt{\delta}\}} = T(x)\mathbf{1}_{\{x\in \interior{\cS}\}},
	\end{align*}
	the remaining steps of the proof follow as in Step 1. This concludes the proof.
\end{proof}

\begin{lemma}
\label{lem:wasserstein.integrals.converge}
	Let $p\in[1,\infty)$, let $a_0\in \mathcal{A}$ and assume that $f$ is continuous and, for some constant $c>0$, satisfies $|f(x,a)|\leq c(1+|x|^p)$ for all $x\in\mathcal{S}$ and all $a$ in a neighborhood of $a_0$.
	Let $(\mu_n)_{n\in\mathbb{N}}$ be a sequence of probability measures which converges to some $\mu$ w.r.t.\ $W_p^{|\cdot|}$ and $(a_n)_{n\in\mathbb{N}}$ be a sequence which converges to $a_0$.
	Then $\int_{\mathcal{S}} f(x,a_n)\,\mu_n(dx)\to\int_{\mathcal{S}} f(x,a_0)\,\mu(dx)$ as $n\to\infty$.
\end{lemma}
\begin{proof}
	Let $K$ be a small neighborhood of $a_0$ such that $|f(x,a)|\leq c(1+|x|^p)$ for all $x\in\mathcal{S}$ and $a\in K$.
	The measures $\mu_n\otimes \delta_{a_n}$ converge in $W_p^{|\cdot|}$ to the measure $\mu\otimes \delta_{a_0}$.
	As $\int_{\mathcal{S}} f(x,a_n)\,\mu_n(dx)=\int_{\mathcal{S}\times K} f(x,a)\,(\mu_n\otimes\delta_{a_n})(d(x,a))$ and similarly for $\mu\otimes\delta_{a_0}$, the claim follows from \cite[Lemma 4.3, p.43]{villani2008optimal}.
\end{proof}

The following lemma relates to the financial economics applications described in \cite{main}. We focus on a sufficient condition for the second part of Assumption \ref{ass:main}. For this, we assume that $\mu$ does not contain any redundant assets, i.e.\ $\mu(\{ x\in\mathbb{R}^d : \langle a , x-x_0\rangle >0\})>0$ for every $a\neq 0$. If $\mu$ satisfies this condition, we call it non-degenerate. Note that this condition is slightly stronger than no-arbitrage. However, if $\mu$ satisfies no arbitrage, then one can always delete the redundant dimensions in $\mu$ similarly to the remark after Theorem \ref{thm.constraints}, so that the modified measure satisfies $\mu(\{ x\in\mathbb{R}^d : \langle a , x-x_0\rangle >0\})>0$ for every $a\neq 0$.

\begin{lemma}
\label{lem:optimal.strategies.converge}
   Assume that $l\colon\R \to \R$ is convex, increasing, bounded from below and $f(x,a):= l(g(x)+\langle a,x\rangle)$ satisfies the first part of Assumption \ref{ass:main}.
   Furthermore assume that $\mu$ is non-degenerate in the above sense.
   Then for every $\delta\geq 0$ there exists an optimizer $a^\opt_\delta\in\mathbb{R}^d$ for $V(\delta)$, i.e.,
   \[V(\delta)=\sup_{\nu\in B_{\delta}(\mu)} \int_{\R^d} l(g(x)+\langle a^\opt_\delta, x-x_0 \rangle )\,\nu(dx)<\infty.\]
	Furthermore, if $l$ is strictly convex, the optimizer $a^\opt$ of $V(0)$ is unique and $a^\opt_\delta\to a^\opt$ as $\delta\to0$. In particular, Assumption \ref{ass:main} is satisfied.
\end{lemma}
\begin{proof}
    The first statement is trivially true if $l$ is constant, so assume otherwise in the following.
    Moreover, note by the first part of Assumption \ref{ass:main} we have $V(\delta)<\infty$ for all $\delta\geq 0$.
	Now fix $\delta\geq 0$, and let $(a_n)_{n\in\mathbb{N}}$ be a minimizing sequence, i.e.\
	\[V(\delta)=\lim_{n\to\infty} \sup_{\nu\in B_{\delta}(\mu)} \int_{\R^d} l(g(x)+\langle a_n, x-x_0 \rangle )\,\nu(dx).\]
    If $(a_n)_{n\in\mathbb{N}}$ is bounded, then after passing to a subsequence there is a limit, and Fatou's lemma shows that this limit is a minimizer.
    It remains to argue why $(a_n)_{n\in\mathbb{N}}$ is bounded.
    Heading for a contradiction, assume that $|a_n|\to\infty$ as $n\to\infty$.
	After passing to a (not relabeled) subsequence, there is $\tilde{a}\in\mathbb{R}^d$ with $|\tilde{a}|=1$ such that $a_n/|a_n|\to \tilde{a}$ as $n\to\infty$.
	By our assumption we have $\mu(\{ x\in\mathbb{R}^d: \langle \tilde{a} ,x-x_0\rangle > 0 \})>0$.
	As  $l$ is bounded below this shows that
	\[\sup_{\nu \in B_{\delta}(\mu)} \int_{\R^d} l(g(x)+\langle a_n , x-x_0 \rangle ) \,\nu(dx)
	\geq \int_{\R^d} l(g(x)+\langle a_n , x-x_0 \rangle ) \,\mu(dx) \to \infty,\]
	as $n\to\infty$, a contradiction. 
	
    To prove the second claim note that strict convexity of $l$ readily implies that $V(0)$ admits a unique minimizer $a^\opt$.  
    Now, heading for a contraction, assume that there exists a subsequence $(\delta_n)_{n\in \N}$ converging to zero, such that $a^\opt_{\delta_n}$ does not converge to $a^\opt$. 
    The exact same reasoning as above shows that $(a^\opt_{\delta_n})_{n\in\mathbb{N}}$ is bounded, hence (possibly after passing to a not relabeled subsequence) there is a limit $\tilde{a}\neq a^\opt$. Using Fatou's lemma once more implies
    \begin{align*}
    V(0)
    &<\int_{\R^d} l(g(x)+\langle \tilde{a}, x-x_0 \rangle )\,\mu(dx) \\
    &\leq \liminf_{n\to \infty} \int_{\R^d} l(g(x)+\langle a^\opt_{\delta_n}, x-x_0 \rangle )\,\mu(dx) 
    \leq \liminf_{n\to \infty} V(\delta_n).
    \end{align*}
   	On the other hand, plugging $a^\opt$ into $V(\delta)$ implies
    \[\limsup_{n\to \infty} V(\delta_n)
    \leq \limsup_{n\to \infty} \sup_{\nu\in B_{\delta_n}(\mu) } \int_{\R^d} l(g(x)+\langle a^\opt, x-x_0 \rangle )\,\nu(dx)
    =V(0),\]    
    which follows from as $l(g(x)+\langle a^\opt, x-x_0 \rangle )\leq c(1+|x|^p)$ and that any $\nu_n\in B_{\delta_n}(\mu)$ converges in $W^{|\cdot|}_p$ to $\mu$ by definition.
    This gives the desired contraction.
\end{proof}

In analogy to the above result, the following summarizes simple sufficient conditions for the second part of Assumption \ref{ass:main}.
\begin{lemma}\label{lem:coercive}
Assume that either $\mathcal{A}$ is compact or that $a\mapsto V(0,a)$ is coercive, in the sense that $V(0,a_n)\to \infty$ if $|a_n|\to \infty$. Moreover, assume that $f$ is continuous, such that $f(x,a)\le c(1+|x|^p)$ for some $c\ge 0$. Then the second part of Assumption \ref{ass:main} is satisfied.
\end{lemma}
\begin{proof}
Let us first note that for fixed $\delta\ge 0$ the function $a\mapsto V(\delta, a)$ is lower semiconinuous as a supremum of continuous functions $a\mapsto \int f(x,a)\,\nu(dx)$ for $\nu\in B_\delta(\mu)$. Next we note that $\mathcal{A}^\opt(\delta)\neq \emptyset$. Indeed, if $\mathcal{A}$ is compact, this directly follows from lower semicontinuity of $a\mapsto V(\delta, a)$. Otherwise, the fact that $V(\delta, a)\ge V(0,a)$ for all $a\in \mathcal{A}$ and coercivity imply that any minimising sequence $(a_n)_{n\in \N}$ is bounded. Lastly, we show that any accumulation point of such a sequence is an element of $\mathcal{A}^\opt(0)$. By the above we can assume (by taking a subsequence without relabelling if necessary)  that $\lim_{n\to \infty} a_n=a\in \mathcal{A}$. If $a\notin \mathcal{A}^\opt$, then 
\begin{align*}
\liminf_{n\to \infty} V(\delta_n, a_n)\ge \lim_{n\to \infty} V(0, a_n)=V(0,a)>V(0,a^\opt)=\lim_{n\to \infty} V(\delta_n, a^\opt)
\end{align*}
for any $a^\opt\in \mathcal{A}^\opt(0)$. This contradicts $a_n \in \mathcal{A}^\opt(\delta_n)$ for all $n\in \N$ and concludes the proof.
\end{proof}

\begin{proof}[Proof of Corollary \ref{Cor:sensitivity}]
We start with the ``$\leq$"-inequality.
First, note that for any $\delta>0$, $a^r\in \Aoptim{r}$, and $\nu^{r+\delta}\in B_{r+\delta}^\opt(\mu,a^r)$, we have 
\begin{align*}
V(r+\delta)
&\leq V(r+\delta,a^r)
=\int_{\mathcal{S}} f(x,a^r)\,\nu^{r+\delta}(dx) , \\
V(r)
& \geq  \sup_{\nu \in B_{r}(\mu)\cap B_\delta(\nu^{r+\delta}) } \int_{\mathcal{S}} f(x,a^r)\,\nu(dx) .
\end{align*}
This implies that
\begin{align}
V(r+\delta)-V(r)
&\le \sup_{\pi \in C_{\delta}(\nu^{r+\delta})} \int_{\mathcal{S}\times \mathcal{S}} f(x,a^r)- f(y,a^r)\,\pi(dx,dy)
\nonumber \\
&= \sup_{\pi \in C_{\delta}(\nu^{r+\delta})} \int_0^1 \int_{\mathcal{S}\times \mathcal{S}} \langle \nabla_x f(y+t(x-y),a^r),(x-y) \rangle\,\pi(dx,dy)\,dt 
\nonumber \\
&\le \delta \sup_{\pi \in C_{\delta}(\nu^{r+\delta})} \int_0^1\left(\int_{\mathcal{S}\times \mathcal{S}}  \|\nabla_x f(y+t(x-y),a^r)\|^q\,\pi(dx,dy)\right)^{1/q}\,dt.
\label{eq:Vrplusdelta.something}
\end{align}
Note that the assumption $|\nabla_x f(x, a)|\le c (1+|x|^{p-1-\epsilon})$ implies $|\nabla_x f(x, a)|^q\le c (1+|x|^{\frac{p(p-1-\epsilon)}{p-1}})$ ( for some new constant $c$).
To simplify notation let us thus define $\tilde{\epsilon}=(p-1-\epsilon)/(p-1)<1$
and recall that $B_{r+1}(\mu)$ is compact w.r.t.\ $W^{|\cdot|}_{p\tilde{\epsilon}}$ by Lemma \ref{lem:ball.compact}, hence there is $\tilde{\nu}^{r}\in B_{r}(\mu)$ such that (after passing to a subsequence) $\nu^{r+\delta}\to \tilde{\nu}^{r}$ w.r.t.\ $W^{|\cdot|}_{p\tilde{\epsilon}}$ as $\delta\to0$.
The same arguments as in the proof of Theorem \ref{thm:main} show that \eqref{eq:Vrplusdelta.something} (divided by $\delta$) converges to $\left( \int_{\mathcal{S}} \left\| \nabla_x f(x,a^{r}) \right\|^q \tilde{\nu}^{r}(dx) \right)^{1/q}$ when $\delta\to0$.
So, to conclude the ``$\leq$''-part, all that is left to do is show that $\tilde{\nu}^r\in B_r^\opt(\mu,a^r)$, which follows as 
\[V(r)
\leq \lim_{\delta\to 0} V(r+\delta)
\leq \lim_{\delta\to 0} \int_{\mathcal{S}} f(x,a^{r})\nu^{r+\delta}(dx)
= \int_{\mathcal{S}} f(x,a^r)\, \tilde{\nu}^r(dx)
\leq V(r).\]

We now turn to the proof of the ``$\geq$"-inequality.
To that end, let $(a^{r+\delta})_{\delta>0}$ be a sequence of optimizers, i.e. $a^{r+\delta} \in \Aoptim{r+\delta}$ for all $\delta>0$. 
Then by assumption there exists $a^{r} \in \Aoptim{r}$ such that (after passing to a subsequence) $\lim_{\delta \to 0} a^{r+\delta}=a^{r}$. 
Let $\nu^r\in B_r^\opt(\mu, a^r)$ be arbitrary.
As  $B_{\delta}(\nu^r)\subset B_{r+\delta}(\mu) $ (by the triangle inequality) we have 
\begin{align*}
V(r+\delta)
&\geq \sup_{\nu\in B_{\delta}(\nu^r)} \int_{\mathcal{S}} f(x,a^{r+\delta})\,\nu(dx).
\end{align*}
As further (trivially) $V(r)\le \int_{\mathcal{S}} f(x,a^{r+\delta})\,\nu^r(dx)$ we conclude
\begin{align*}
\frac{V(\delta+r)-V(r)}{\delta}
&\ge\sup_{\nu\in B_{\delta}(\nu^r)} \frac{1}{\delta} \int_{\mathcal{S}} f(x,a^{r+\delta})\nu(dx)-\int_{ \mathcal{S}} f(x,a^{r+\delta})\nu^{r}(dx) \\
&\to  \left( \int_{\mathcal{S}}\left\|\nabla_x f(x,a^{r}) \right\|^q\nu^{r}(dx) \right)^{1/q},
\end{align*}
as $\delta\to0$, where the the last equality follows from the exact same arguments as presented int he proof of Theorem \ref{thm:main}.
As $\nu^r\in B_r^\opt(\mu, a^r)$ was arbitrary, the claim follows.
\end{proof}

\begin{proof}[Proof of Theorem \ref{thm.constraints}]
	We start by showing the easier estimate
    \begin{align}
    \label{eq:limsup.estimate.constraints}
   	\begin{split}
   	& \limsup_{\delta\to 0}\frac{V^\Phi(\delta)-V^\Phi(0)}{\delta}\\
    &\leq \inf_{a^\opt \in \Aoptim{0}}\inf_{\lambda \in\mathbb{R}^m} \left(\int_{\mathcal{S}} \Big\| \nabla_x f(x,a^\opt)+\sum_{i=1}^m \lambda_i\nabla_x\Phi_i(x)\Big\|^q \,\mu(dx)\right)^{1/q}.
    \end{split}
    \end{align}
    To that end, let $a^\opt\in \Aoptim{0}$ and $\lambda \in\mathbb{R}^m$ by arbitrary.
    Then $V^\Phi(0)=\int_{\mathcal{S}} f(x,a^\opt)+\sum_{i=1}^m \lambda_i\Phi_i(x)\,\mu(dx)$.
    Moreover, as $B_\delta^\Phi(\mu)\subset B_\delta(\mu)$, it further follows that  $V^\Phi(\delta)\leq \sup_{\nu\in B_\delta(\mu)}\int_{\mathcal{S}} f(y,a^\opt)+\sum_{i=1}^m \lambda_i\Phi_i(y)\,\nu(dy)$.
    Therefore \eqref{eq:limsup.estimate.constraints} is a consequence of Theorem \ref{thm:main} (applied to the function $\tilde{f}(x,a):=f(x,a^\opt)+\sum_{i=1}^m \lambda_i\Phi_i(x)$).
 
    To show the other direction, i.e.\ that
    \begin{align} 
    \label{eq:liminf.estimate.constraints}
   	\begin{split}
   	&\liminf_{\delta\to 0}\frac{V^\Phi(\delta)-V^\Phi(0)}{\delta}\\
	&\geq \inf_{a^\opt \in \Aoptim{0}}\inf_{\lambda \in\mathbb{R}^m} \left(\int_{\mathcal{S}\times \mathcal{S}} \Big\| \nabla_x f(x,a^\opt)+\sum_{i=1}^m \lambda_i\nabla_x \Phi_i(x)\Big\|^q \,\mu(dx)\right)^{1/q}.
	\end{split}
    \end{align}	
    pick a (not relabeled) subsequence of $(\delta)_{\delta>0}$ which converges to the liminf.
    For $a^\opt_\delta\in \Aoptim{\delta}$, there is another (again not relabeled) subsequence which converges to some $a^\opt\in \Aoptim{0}$.
    From now on stick to this subsequence.
    In a first step, notice that 
    \begin{align}
    V^{\Phi}(\delta)
    &=\sup_{\nu\in B_{\delta}(\mu)}\inf_{\lambda\in\mathbb{R}^m}\int_{\mathcal{S}}  f(y,a^\opt_\delta)+\sum_{i=1}^m \lambda_i \Phi_i(y) \,\nu(dy) \nonumber \\
    &=\inf_{\lambda\in\mathbb{R}^m}\sup_{\nu\in B_{\delta}(\mu)}\int_{\mathcal{S}} f(y,a^\opt_\delta)+\sum_{i=1}^m \lambda_i \Phi_i(y)  \,\nu(dy). \label{eq:formulaminimax}
    \end{align} 
    Indeed, this follows from a minimax theorem (see \cite[Cor.\ 2, p.\ 411]{min_max_terkelsen1972}) and appropriate compactness of $B_\delta(\mu)$ as stated in Lemma \ref{lem:ball.compact}.
    For notational simplicity let $\lambda^\opt_\delta$ be an optimizer for \eqref{eq:formulaminimax}.
    Then
	\begin{align}
	\label{eq:estimate.constraings.below}
	\begin{split}
   	&\frac{V^\Phi(\delta)-V^\Phi(0)}{\delta}\\
   	&\geq \frac{1}{\delta} \sup_{\pi\in C_\delta(\mu)}\int_{\mathcal{S}\times \mathcal{S}} f(y,a^\opt_\delta)-f(x,a^\opt_\delta)+\sum_{i=1}^m \lambda^\opt_{\delta,i}(\Phi_i(y)-\Phi_i(x)) \,\pi(dx,dy),
   	\end{split}
   	\end{align}
   	where we used that $V^\Phi(0)\leq\int_{\mathcal{S}}  f(x,a^\opt_\delta)+\sum_{i=1}^m\lambda^\opt_{\delta,i} \Phi_i(x)  \,\mu(dx)$.
   	Now, in case that $\lambda^\opt_\delta$ is uniformly bounded for all small $\delta>0$, after passing to a subsequence, it converges to some $\lambda^\opt$.
   	Then it follows from the exact same arguments as used in the proof of Theorem \ref{thm:main} that 
   	\begin{align*}
   	&\liminf_{\delta\to 0}
   	\frac{1}{\delta} \sup_{\pi\in C_\delta(\mu)}\int_{\mathcal{S}\times \mathcal{S}} f(y,a^\opt_\delta)-f(x,a^\opt_\delta)+\sum_{i=1}^m\lambda^\opt_{\delta, i}(\Phi_i(y)-\Phi_i(x)) \,\pi(dx,dy)\\
   	&\geq \Big(\int_{\mathcal{S}} \Big\| \nabla_x f(x,a^\opt)+\sum_{i=1}^m\lambda^\opt_i\nabla_x\Phi_i(x)\Big\|^q \,\mu(dx)\Big)^{1/q}
   	\end{align*}
   	which shows \eqref{eq:liminf.estimate.constraints}.
   	It remains to argue why $\lambda^\opt_\delta$ is bounded for small $\delta>0$.
    By \eqref{eq:estimate.constraings.below} and the estimate ``$\sup (A+B)\geq \sup A + \inf B$'' we have
    \begin{align*}
   \frac{V^\Phi(\delta)-V^\Phi(0)}{\delta}\
    &\geq \frac{1}{\delta} \sup_{\pi\in C_{\delta}(\mu)}\int_{\mathcal{S}\times \mathcal{S}} \sum_{i=1}^m \lambda^\opt_{\delta, i}(\Phi_i(y)-\Phi_i(x)) \,\pi(dx,dy) \\
    &\quad+\frac{1}{\delta}  \inf_{\pi\in C_{\delta}(\mu)}\int_{\mathcal{S}\times \mathcal{S}} f(y,a^\opt_\delta)-f(x,a^\opt_\delta)\,\pi(dx,dy).
    \end{align*}
    The second term converges to $-(\int_{\mathcal{S}}\|\nabla_x f(x,a^\opt)\|^q\,\mu(dx))^{1/q}$ (see the proof of Theorem \ref{thm:main}), in particular it is bounded for all $\delta>0$ small.
    On the other hand by \eqref{eq:non-degenerate} and continuity as well as growth of $x\mapsto \nabla_x\Phi_i(x)$, the first term is larger than $c|\lambda^\opt_\delta|$ for some $c>0$.
    By \eqref{eq:limsup.estimate.constraints} this implies that $(\lambda^\opt_\delta)_{\delta >0}$ must be bounded for small $\delta>0$.
\end{proof}

We have used the following lemma:
\begin{lemma}
\label{lem:ball.compact}
    Let $p,q\in[1,\infty)$ such that $q<p$ and let $\mu$ be a probability measure on $\mathcal{S}$.
    Then $p$-Wasserstein ball $B_\delta(\mu)$ is compact w.r.t.\ $W_{q}^{|\cdot|}$.
   \end{lemma}
\begin{proof}
We recall that $\|\cdot\|_\ast$ is lower semicontinuous and there exists $c>0$ such that $|x|\le c\|x\|_\ast$ for all $x\in \mathbb{R}^d$.
    As $\int_\mathcal{S} |x|^p\, \mu(dx)<\infty$ by assumption, an application of Prokhorov's theorem shows that $B_{\delta}(\mu)$ is weakly precompact (recall the convention that we continuity is defined for $(\R^d,|\cdot|)$).
    Hence, for every sequence of measures $(\nu_n)_{n\in \N}$ in $B_{\delta}(\mu)$ there exists a subsequence, which we also call $(\nu_n)_{n\in \N}$ and a measure $\nu$ such that $\nu_n$ converges weakly to $\nu$.
    As $W_p$ is weakly lower semicontinuous (see \cite[Lemma 4.3, p.43]{villani2008optimal}), this implies $\nu\in B_{\delta}(\mu)$. 
	Applying the same argument to the tight sequence $(\tilde{\nu}_n)_{n\in\N}$ defined via 
	\[ \tilde{\nu}_n(dx):=  \frac{ |x|^q }{ \int_{\mathcal{S}} |y|^q\,\nu_n(dy)} \nu_n(dx)\]
	we conclude that there exists another subsequence of $(\nu_n)_{n\in \N}$ which also converges in $W_q^{|\cdot|}$. This concludes the proof.
\end{proof}

\section{Discussion, proofs and auxiliary results related to Theorem \ref{thm:sens}}

\subsection{Further discussion of Theorem \ref{thm:sens}}

We note that a natural way to compute the sensitivity of $a^\opt_\delta$ would be by combining Theorem \ref{thm:main} with chain rule and  differentiation of the function $V(a,\delta)$. This cannot however be rigorously justified as the following remark demonstrates.
\begin{remark}
Let us point out that it is not true that $a \mapsto V(a, \delta)$ is differentiable for $\delta>0$ under the sole assumption that $(x,a)\mapsto f(x,a)$ is sufficiently smooth and $\nabla_a^2 f\neq 0$.

To give an example, let $\mathcal{S}=\R$, $\|\cdot\|=|\cdot|$, $\mathcal{A}=\R$ and take $f(x,a):=ax+a^2$ and $\mu=\delta_0$. 
A quick computation shows $V(\delta,a)= \delta |a|+a^2$ (independently of $p$).
In particular $V(\delta)=0$ and $a^\opt_{\delta}=a^\opt=0$ for all $\delta> 0$ and  $a\mapsto V(\delta,a)$ is clearly not differentiable in $a=0$.
\end{remark}

Instead, we use a more involved argument, combining  differentiability of $a \mapsto V(0, a)$ with a Lagrangian approach. This however requires slightly stricter growth assumptions than the ones imposed in Assumption \ref{ass:main}, which are specified in Assumption \ref{ass:sens}.

\begin{example}\label{ex:LASSO} We provide detailed computations behind the square-root LASSO/Ridge regression example discussed in \cite{main}. 
We consider $\cA=\R^k$, $\cS=\R^{k+1}$. We fix norms $\|(x,y)\|=|x|_s$, $\|(x,y)\|_{\ast}=|x|_r\mathbf{1}_{\{y=0\}}+\infty \mathbf{1}_{\{y\neq 0\}}$, for some $s>1$, $1/s+1/r=1$ and $(x,y)\in \R^{k}\times \R$. We recall than then \eqref{eq:lasso} holds and we can apply our methodology for $f((x,y),a):=(y-\langle x, a \rangle)^2$. In general we have $$\nabla_{(x,y)} f((x,y), a^\opt)= (-2(y-\langle x, a^\opt \rangle) a^\opt,2(y-\langle x, a^\opt \rangle))$$ 
 $\nabla_a^2V(0, a^\opt)=2D$ and
\begin{align*}
\left( \int_{\R^{k+1}} \|\nabla_{(x,y)} f((x,y), a^\opt)\|^2\, \mu(dx, dy)\right)^{1/2}&=2|a^\opt|_s\left( \int_{\R^{k+1}} (y-\langle x, a^\opt \rangle)^2\, \mu(dx, dy)\right)^{1/2}\\
&= 2|a^\opt|_s \sqrt{V(0)}.
\end{align*}
Recalling the convention that $\nabla_{(x,y)} \nabla_a f\in \R^{k\times (d+1)}$ is given by
\begin{align*}
\begin{bmatrix}
\nabla_{x_1} \nabla_{a_1} f & \dots& \nabla_{x_d} \nabla_{a_1}f &  \nabla_{y} \nabla_{a_1}f\\
\nabla_{x_1}\nabla_{a_2} f& \dots & \nabla_{x_d} \nabla_{a_2}f &  \nabla_{y} \nabla_{a_2}f\\
\vdots & \vdots &\vdots &\vdots\\
\nabla_{x_1}\nabla_{a_k} f& \dots & \nabla_{x_d} \nabla_{a_k}f &  \nabla_{y} \nabla_{a_k}f
\end{bmatrix}
\end{align*}
we conclude 
\begin{align*}
\nabla_{(x,y)} \nabla_a f((x,y),a^\opt)= 2\left(-y\mathbf{I} + x(a^\opt)^T + (\mathbf{I}a^\opt)(\mathbf{I}x),-x\right),
\end{align*}
where $\mathbf{I}$ is the $k\times k$ identity matrix. 
Recall furthermore that $\int_{ \mathbb{R}^{k+1} } (y-\langle a^\opt,x\rangle )x_i\mu(dx,dy)=0$ for all $1\leq i\leq k$ and in particular $V(0)=\int_{ \mathbb{R}^{k+1}} (y^2-\langle a^\opt, x\rangle y)\mu(dx,dy)$. Set now 
$$h((x,y)):=(\text{sign}(x_1)\,|x_1|^{s-1}, \dots, \text{sign}(x_k)\,|x_k|^{s-1},0)\cdot |x|_s^{1-s}.$$  Then $\langle (x,y), h((x,y))\rangle =|x|_s$ and $|h(x,y)|_{r}=1$ for $(x,y)\in \cS\setminus U$. 
As $h$ does not depend on the last coordinate, we also write simply $h(x)$ for $h((x,y))$. 
As $q=2$ we have in particular
\begin{align*}
&\int_{\R^{k+1}} \nabla_{(x,y)}\nabla_a f((x,y), a^\opt) \frac{h(\nabla_{(x,y)} f((x,y), a^\opt))}{ \|\nabla_{(x,y)} f((x,y), a^\opt)\|^{-1}}\, \mu(dx, dy) \\
&=4\int_{\R^{k+1}} \big [-y\mathbf{I} +x(a^\opt)^T+(\mathbf{I}a^\opt)(\mathbf{I}x) \big] \,  h(-(y-\langle x, a^\opt \rangle)a^\opt)\, |a^\opt|_s |y-\langle x, a^\opt \rangle | \, \mu(dx, dy)\\
&=-4 |a^\opt|_s  \int_{\R^{k+1}} \big [-y\mathbf{I} +x(a^\opt)^T+(\mathbf{I}a^\opt)(\mathbf{I}x) \big] \, (y-\langle x, a^\opt \rangle) h(a^\opt) \, \mu(dx, dy)\, \\
&=4 |a^\opt|_s V(0)\, h(a^\opt).
\end{align*}
In conclusion
\begin{align*}
a^\opt_\delta \approx &\ a^\opt -\Big(\int_{\R^{k+1}} \|\nabla_{(x,y)} f((x,y),a^\opt)\|^2\,\mu(dx,dy)\Big)^{-1/2} (\nabla^2_a V(0,a^\opt))^{-1} \\
&\qquad\qquad\qquad\cdot  \int_{\R^{k+1}} \frac{\nabla_{(x,y)}\nabla_a f((x,y),a^\opt)\, h(\nabla_{(x,y)} f((x,y),a^\opt))}{\|\nabla_{(x,y)} f((x,y),a^\opt)\|^{-1}} \, \mu(dx,dy)\cdot \delta\\
&=a^\opt - \frac{1}{4|a^\opt|_s \sqrt{V(0)}} \, D^{-1}\, 4 |a^\opt|_s V(0)\, h(a^\opt)\cdot \delta \\
&=a^\opt-  \sqrt{V(0)} D^{-1}\, h(a^\opt)\cdot \delta.
\end{align*}
Let us now specialise to the typical statistical context and let $\mu=\mu_N$ equal to the empirical measure of $N$ data samples, i.e., $\mu_N=\frac{1}{N}\sum_{i=1}^N \delta_{(x_i,y_1)}$ for some points $x_1, \dots, x_N\in \R^d$ and $y_1,\ldots,y_N\in \R$. Let us write $x_i=(x_{i,1}, \dots, x_{i, d})$ and $X=(x_{i,j})_{i=1, \dots, N}^{j=1, \dots, d}$. Then in particular 
\begin{align*}
D=\int_{\R^{k+1}} xx^T\, \mu_N(dx,dy)= \frac{1}{N} X^TX
\end{align*}
and we recover the notation common in statistics. In particular, $a^\opt = (X^TX)^{-1} X^T y$. If we now assume that $X^TX=\mathbf{I}$ (and hence $D^{-1}=N \mathbf{I}$), then we can easily compute
\begin{align*}
V(0)&=\frac{1}{N}(y-Xa^\opt)^T(y-Xa^\opt)=\frac{1}{N}(y-XX^Ty)^T(y-XX^Ty)\\
&=\frac{1}{N}y^T(\mathbf{I}-XX^T)^T(\mathbf{I}-XX^T)y=\frac{1}{N}y^T(\mathbf{I}-XX^T-XX^T+XX^TXX^T)y\\
&=\frac{1}{N}y^T(\mathbf{I}-XX^T)y
\end{align*}
Note that, under the assumption that $\sum_{i=1}^N y_i=0$, $R^2$ is defined as
\begin{align*}
R^2&=1-\frac{y^T(\mathbf{I}-XX^T)y}{y^T y}=\frac{y^Ty-y^T(\mathbf{I}-XX^T)y}{y^Ty}=\frac{y^TXX^Ty}{y^Ty}.
\end{align*}
Thus in the case $s=1$ we have
\begin{align*}
a^\opt_\delta &\approx a^\opt-  \sqrt{V(0)} D^{-1}\, \text{sign}(a^\opt)\cdot \delta= a^\opt-\sqrt{N}\, \sqrt{y^Ty- y^TXX^Ty}\,\text{sign}(a^\opt)\cdot \delta\\
&= a^\opt-\sqrt{N}\, \sqrt{y^Ty}\,\sqrt{1- \frac{y^TXX^Ty}{y^Ty}}\,\text{sign}(a^\opt)\cdot \delta\\
&= a^\opt-\sqrt{N}\, |y|\,\sqrt{1-R^2}\,\text{sign}(a^\opt)\cdot \delta.
\end{align*}
Furthermore, in the case $s=2$ we have
\begin{align}\label{eq:jan}
a^\opt_\delta &\approx a^\opt -D^{-1}\frac{\sqrt{V(0)}}{|a^\opt|_2}a^\opt \delta= a^\opt\left(1 -N\frac{\sqrt{y^T(\mathbf{I}-XX^T)y}}{\sqrt{N}|a^\opt|_2 }\,\delta \right) \nonumber\\
&=  a^\opt\left(1 -\frac{\sqrt{N\,y^T(\mathbf{I}-XX^T)y}}{|a^\opt|_2 }\,\delta \right) .
\end{align}
We also have
\begin{align*}
|a|=\sqrt{\langle a^\opt, a^\opt \rangle} =\sqrt{y^T XX^Ty},
\end{align*}
so \eqref{eq:jan} simplifies to 
\begin{align*}
a^\opt_\delta &\approx a^\opt\left(1 -\frac{\sqrt{N\,y^T(\mathbf{I}-XX^T)y}}{\sqrt{y^T XX^Ty}}\,\delta \right) =a^\opt\left(1 -\delta \sqrt{N\left(\frac{y^Ty}{y^TXX^Ty}-1\right)} \right) \\
&= a^\opt\left(1 - \delta\sqrt{N \left(\frac{1}{R^2}-1\right)}\right).
\end{align*}
\end{example}

\begin{remark}
While $|\cdot|_1$ is not strictly convex, the above example can still be adapted to cover this case under the additional assumption, that $a^\opt$ has no entries which are equal to zero. Indeed, we note that $x\mapsto h(x, y)$ is continuous (even constant) at every point $x$ except if a component of $x$ is equal to zero. Thus the proof of Lemma \ref{lem:inf} still applies if we assume that $g$ has $\mu$-a.s. no components which are equal to zero instead of merely assuming that $g\neq 0$ $\mu$-a.s..
\end{remark}

\begin{example}\label{ex:OutOfSample} We provide further details and discussion to complement the out-of-sample error example in \cite{main}. First, we recall the remainder term obtained therein:
\begin{align*}
\Delta_N &:=\Big(\int |\nabla_x f(x,a^{\opt,N})|_s^q\,\mu_N(dx)\Big)^{\frac{1}{q}-1} \cdot \left(\int \nabla_a^2 f(x, a^{\opt,N})\,\mu_N(dx)\right)^{-1} \\
&\quad \cdot \int \frac{\nabla_{x}\nabla_a f(x,a^{\opt,N})\, h(\nabla_x f(x,a^{\opt,N}))}{|\nabla_x f(x,a^{\opt,N})|_s^{1-q}} \, \mu_N(dx)-(\nabla_a^2 V(0, a^{\opt}))^{-1}\Theta,\quad \textrm{where}\\
\Theta & := \Big(\int |\nabla_x f(x,a^\opt)|_s^q\,\mu(dx)\Big)^{\frac{1}{q}-1}\cdot 
   \int  \frac{\nabla_{x}\nabla_a f(x,a^\opt)\, h(\nabla_x f(x,a^\opt))}{|\nabla_x f(x,a^\opt)|_s^{1-q}} \, \mu(dx).
   \end{align*}
Recall that $\mu_N\to \mu$ in $W_p$ holds a.s. We suppose that Assumptions \ref{ass:main} and \ref{ass:sens} hold, and that for any $r>0$, there exists $c>0$ such that the following hold uniformly for all $|a|\leq r$: 
\begin{align}\label{eq:condit}
\begin{split}
\sum_{i=1}^k\Big| \nabla_{a}\nabla_{a_i} f(x,a)\Big| &\le c(1+|x|^{p}), \\
\Bigg| \frac{\nabla_{x}\nabla_a f(x,a^\opt)\, h(\nabla_x f(x,a^\opt))}{|\nabla_x f(x,a^\opt)|^{1-q}}\Bigg| &\le c(1+|x|^{p}).
\end{split}
\end{align}
Recall from \eqref{eq:M-est asymptotics} that we already know that $a^{\opt,N}\to a^{\opt}$ a.s. Under the above integrability assumption, Lemma \ref{lem:wasserstein.integrals.converge} gives
\begin{align*}
&\left| \int \nabla_{a_{i}}\nabla_{a_j} f(x, a^{\opt,N})\,\mu_N(dx) -\int \nabla_{a_{i}}\nabla_{a_j} f(x, a^{\opt})\,\mu(dx) \right|\to 0,
\end{align*}
with analogous convergence for the other two terms in $\Delta_N$. We conclude that $\Delta_N\to 0$ a.s. and that \eqref{eq:CLTfromOSE} and \eqref{eq:out-of-sample-error} hold.

We now show how the arguments above can be adapted to extend and complement \cite[Prop. 17]{anderson2019improving}. Therein, the authors study VRS$(\delta)$ which is the expectation over realisations of $\mu_N$ of 
\begin{align*}
\int f(x,a^{\opt,N})\,\mu(dx)-\int f(x,a^{\opt,N}_\delta)\,\mu(dx).
\end{align*}
If VRS$(\delta)>0$ then, on average, the robust problem offers an improved performance, i.e., finds a better approximation to the true optimizer $a^\opt$ than the classical non-robust problem. 
If we work with the difference above, then we look at first order Taylor expansion and obtain
\begin{align*}
V(0,a^{\opt,N}_\delta)-V(0,a^{\opt,N}) = \nabla_a V(0,a^{\opt,N})(a^{\opt,N}_\delta-a^{\opt,N})+o(|a^{\opt,N}_\delta-a^{\opt,N}|),
\end{align*}
which holds under the first condition in \eqref{eq:condit}. This can be compared with \cite[Lemma 1]{anderson2019improving} which was derived under a Lipschitz continuity assumption on $a\mapsto f(x,a)$. 
For the quadratic case of \cite[Prop. 17]{anderson2019improving} we have $f(x,a)=1/2 a^2-g(x)a$, where we took $d=1$ for notational simplicity. We then have $\nabla_x f(xa)=-g'(x)a,$ $\nabla_a^2 f(x,a)=1$ and $\nabla_{x}\nabla_a f(x,a)\nabla_x f(x,a) =(g'(x))^2 a$. Specialising \eqref{eq:Nsamples opt diff fin} to this setting, with $s=2$, gives
\begin{align*}
a^{\opt,N}_\delta-a^{\opt,N} 
&\approx -\left(\nabla_{a}^{2} V\left(0, a^{\opt}\right)\right)^{-1}\left(\int\left|\nabla_{x} f\left(x, a^{\opt}\right)\right|^{q} \mu_N(d x)\right)^{1/q-1}\cdot \int \frac{\nabla_{x} \nabla_{a} f\left(x, a^{\opt}\right) \nabla_{x} f\left(x, a^{\opt}\right)}{|\nabla_{x} f\left(x, a^{\opt}\right)|^{2-q}} \mu_N(d x)\\
&=-|a^{\opt}|^{1-q} \left(\int |g'(x)|^q\,\mu_N(dx)\right)^{1/q-1} \int \frac{(g'(x))^2 a^{\opt}}{|g'(x)a^{\opt}|^{2-q}}\,\mu_N(dx)\\
&= -\text{sign}(a^{\opt})\left(\int |g'(x)|^q\,\mu_N(dx)\right)^{1/q}.
\end{align*}
While our results work for $p>1$, see Remark \ref{rk:failforp1}, we can formally let $q\uparrow \infty$. The last term then converges to $-\text{sign}(a^{\opt})\|g'\|_{L^\infty(\mu)}$ which recovers \cite[Prop.\ 17]{anderson2019improving}, taking into account that $\text{sign}(a^{\opt})=\text{sign}\left(\int g(x)\,\mu(dx)\right).$

\end{example}

\subsection{Proofs and auxiliary results related to Theorem \ref{thm:sens}}

\begin{lemma}
\label{lem:grad_hedge_new}
	Let $f\colon\mathcal{S}\times\mathcal{A}\to\mathbb{R}$ be differentiable such that $(x,a)\mapsto \nabla_a f(x,a)$ is continuous, fix $a\in\interior{\cA}$, and assume that for some $\varepsilon>0$ we have that $|\nabla_x f(x,\tilde{a})|\leq c(1+|x|^{p-1-\varepsilon})$ and $|\nabla_a f(x,\tilde{a})|\leq c(1+|x|^{p-\varepsilon})$ for some $c>0$, all $x\in\mathcal{S}$ and all $\tilde{a}\in\mathcal{A}$ close to $a$. 
	Further fix $\delta\geq 0$ and recall that $\Boptim_\delta(\mu,a)$ is the set  of maximizing measures given the strategy $a$. 
	Then the (one-sided) directional derivative of $V(\delta, \cdot)$ at $a$ in the in direction $b\in\mathbb{R}^k$ is given by
	\[ \lim_{h\to 0} \frac{V(\delta,a+h b)-V(\delta,a)}{h}
	=\sup_{\nu\in \Boptim_\delta(\mu,a)}\int_{\mathcal{S}} \langle \nabla_a f(x,a), b\rangle\,\nu(dx) . \]
\end{lemma}

\begin{proof}
	Fix $b\in\mathbb{R}^k$. 
	We start by showing that
	\begin{align}
	\label{eq:diff.parameter.geq.inequality}
	\liminf_{h\to 0} \frac{ V(\delta,a+hb)-V(\delta,a) }{h}
    &\geq \sup_{\nu\in \Boptim_\delta(\mu,a)}\int_{\mathcal{S}} \langle \nabla_a f(x,a), b\rangle\,\nu(dx).
	\end{align}
	To that end, let $\nu\in \Boptim_\delta(\mu,a)$ and $h>0$ be arbitrary.
	By definition of $\Boptim_\delta(\mu,a)$ one has $V(\delta,a)=\int_{\mathcal{S}} f(x,a)\,\nu(dx)$.
	Moreover $\Boptim_\delta(\mu,a)\subseteq B_\delta(\mu)$ implies that $V(\delta,a+hb)\geq \int_{\mathcal{S}} f(x,a+hb)\,\nu(dx)$.
	Note that the assumption $|\nabla_x f(x,\tilde{a})|\leq c(1+|x|^{p-1-\varepsilon})$  implies
    \begin{align*}
    |f(x, \tilde{a})-f(0, \tilde{a})|&=\left| \int_0^1 \langle \nabla_x f(tx, \tilde{a}) , x \rangle dt \right|\\
    &\le \int_0^1 c(1+|tx|^{p-1-\epsilon})|x| dt \le c(1+|x|^{p-\epsilon}\vee |x|).
    \end{align*}
	Therefore, by dominated convergence, one has 
    \begin{align*}  
    \liminf_{h\to 0} \frac{V(\delta,a+hb)-V(\delta,a)}{h}
    &\ge \liminf_{h\to 0} \int_{\mathcal{S}} \frac{ f(x, a+hb)-f(x,a) }{h} \,\nu(dx)\\
    &= \int_{\mathcal{S}} \lim_{h\to 0} \frac{ f(x, a+hb)-f(x,a) }{h}\,\nu(dx)\\
    &=\int_{\mathcal{S}} \langle \nabla_a f(x,a), b\rangle \,\nu(dx)
    \end{align*}
    and as $\nu\in \Boptim_\delta(\mu,a)$ was arbitrary, this shows \eqref{eq:diff.parameter.geq.inequality}.
    
    We proceed to show that
    \begin{align}
	\label{eq:diff.parameter.leq.inequality}
	\limsup_{h\to 0} \frac{ V(\delta,a+hb)-V(\delta,a) }{h}
    &\leq \sup_{\nu\in \Boptim_\delta(\mu,a)}\int_{\mathcal{S}} \langle \nabla_a f(x,a), b\rangle\,\nu(dx).
	\end{align}
	For every sufficiently small $h>0$ let $\nu^h\in B_\delta^\opt(\mu,a+hb)$ such that $V(\delta,a+hb)=\int_{\mathcal{S}} f(x,a+hb)\,\nu^h(dx)$.
	The existence of such $\nu^h$ is guaranteed  by Lemma \ref{lem:ball.compact}, which also guarantees that (possibly after passing to a subsequence) there is $\tilde{\nu}\in B_\delta(\mu)$ such that $\nu^h\to\tilde{\nu}$ in $W^{|\cdot|}_{p-\varepsilon}$.
	We claim that $\tilde{\nu} \in \Boptim_\delta(\mu,a)$.
	By Lemma \ref{lem:wasserstein.integrals.converge} one has
	\[ \lim_{h\to 0}V(\delta,a+hb)
	=\int_{\mathcal{S}} f(x,a)\,\tilde{\nu}(dx)
	\leq V(\delta,a).\]
	On the other hand, for any choice $\tilde{\nu}\in \Boptim_\delta(\mu,a)$ one has 
	\[\lim_{h\to 0} V(\delta,a+hb)
	\geq\lim_{h\to 0}\int_{\mathcal{S}}  f(x,a+hb)\,\tilde{\nu}(dx)
	=\int_{\mathcal{S}}  f(x,a)\,\tilde{\nu}(dx)
	=V(\delta,a).\]
	This implies $V(\delta,a)=\int_{\mathcal{S}} f(x,a)\,\tilde{\nu}(dx)$ and in particular $\tilde{\nu}\in \Boptim_\delta(\mu,a)$. 
	At this point expand 
	\[f(x,a+hb)=f(x,a)+\int_0^1 \langle \nabla_a f(x,a+thb),hb\rangle \,dt\]
	so that 
	\begin{align*}  
    &V(\delta,a+hb)-V(\delta,a)\\
    &=\int_{\mathcal{S}} \Big( f(x,a)+\int_0^1 \langle \nabla_a f(x,a+thb),hb\rangle\,dt\Big)\,\nu^h(dx)- \int_{\mathcal{S}} f(x,a)\,\tilde{\nu}(dx)\\
    &\leq\int_{\mathcal{S}} \int_0^1 \langle \nabla_a f(x,a+thb),hb\rangle\,dt\,\nu^h(dx)
    \end{align*}
    where we used $\tilde{\nu}\in \Boptim_\delta(\mu,a)$ for the last inequality.
    Recall that $\nu^h$ converges to $\tilde{\nu}$ in $W^{|\cdot|}_{p-\varepsilon}$ and by assumption $|\nabla_a f(x,\tilde{a})|\leq c(1+|x|^{p-\varepsilon})$ for all $\tilde{a}\in \mathcal{A}$ close to $a$.
    In particular $$\frac{1}{h}\langle \nabla_a f(x,a+thb),hb\rangle\le | \nabla_a f(x,a+thb)| |b|\le c(1+|x|^{p-\varepsilon})$$ for $h$ sufficiently small. 
    As furthermore $(x,a)\mapsto \nabla_a  f(x,a)$ is continuous, we conclude by Lemma \ref{lem:wasserstein.integrals.converge} that 
    \begin{align*}
    \lim_{h \to 0}\frac{1}{h}\int_{\mathcal{S}}\langle \nabla_a f(x,a+thb),hb\rangle\,dt\,\nu^h(dx)=\int_{\mathcal{S}} \langle \nabla_a f(x,a),b\rangle \,\tilde{\nu}(dx).
    \end{align*}
   Lastly, by  by Fubini's theorem and dominated convergence (in $t$)
    \[\frac{1}{h}\int_{\mathcal{S}} \int_0^1 \langle \nabla_a f(x,a+thb),hb\rangle\,dt\,\nu^h(dx)
    \to\int_{\mathcal{S}} \langle \nabla_a f(x,a),b\rangle \,\tilde{\nu}(dx)\]
    as $h\to 0$, which ultimately shows \eqref{eq:diff.parameter.leq.inequality}.
\end{proof}

\begin{lemma}
\label{lem:inf}
Let $q\in (1, \infty)$ and let $f,g\colon\mathcal{S}\to\mathbb{R}^d$ be measurable such that $\int_\mathcal{S} \|f(x)\|^q+\|g\|^q\,\mu(dx)<\infty$ and such that $g\neq 0$ $\mu$-a.s..
Then we have that 
\begin{align}
\label{eq:something}
\begin{split}
&\inf_{\lambda \in \R} \left(\left(\int_{\mathcal{S}} \|f(x)+\lambda g(x)\|^q \, \mu(dx) \right)^{1/q}-\lambda \left(\int_{\mathcal{S}} \|g(x)\|^q\,\mu(dx) \right)^{1/q}\right)\\
&= \int_{\mathcal{S}} \frac{\langle f(x),h(g(x))\rangle }{\|g(x)\|^{1-q} }\,\mu(dx) \cdot \Big(\int_{\mathcal{S}} \|g(x)\|^q\,\mu(dx)\Big)^{1/q-1},
\end{split}
\end{align}
where $h\colon \mathbb{R}^d\setminus\{0\} \to \R^d$ was defined in Lemma \ref{dani:2}.
\end{lemma}

\begin{proof}
	First recall that $h$ is continuous and satisfies $\|x\|=\langle x, h(x) \rangle $ for every $x\neq 0$.
	Now define 
	\[G(x):= \frac{h(g(x))}{\|g(x)\|^{1-q}} \Big(\int_{\mathcal{S}} \|g(z)\|^q\,\mu(dz)\Big)^{1/q-1}
	\quad\text{for }x\in\mathcal{S}.\]
	Similarly, define $G^\lambda$ by replacing $g$ in the definition of $G$ by $g^\lambda:= f + \lambda g$.
	As in the proof of Theorem \ref{thm:main} we compute 
	\[ \int_\mathcal{S} \|G(x)\|_\ast^p \,\mu(dx)=1
	\quad\text{and}\quad
	\left(\int_{\mathcal{S}} \|g(x)\|^q\,\mu(dx) \right)^{1/q}= \int_{\mathcal{S}} \langle g(x),G(x)\rangle \,\mu(dx). \]
	This remains true when $g$ and $G$ are replaced by $g^\lambda$ and $G^\lambda$, respectively.
	Moreover, H\"older's inequality implies that
	\begin{align*}
	\left(\int_{\mathcal{S}} \|g^\lambda(x)\|^q\,\mu(dx) \right)^{1/q} 
	&\geq \int_{\mathcal{S}} \langle g^\lambda (x),G(x)\rangle \,\mu(dx),\\
	\left(\int_{\mathcal{S}} \|g(x)\|^q\,\mu(dx) \right)^{1/q} 
	&\geq \int_{\mathcal{S}} \langle g (x),G^\lambda(x)\rangle \,\mu(dx).	
	\end{align*}
	The first of these two inequalities immediately implies that the left hand side in \eqref{eq:something} is larger than the right hand side.
	
	To show the other inequality, note that $h$ is continuous and satisfies $h(\lambda x)=h(x)$ for $\lambda>0$, hence $h(g(x)) = \lim_{\lambda\to\infty} h(g^\lambda(x))$ for all $x\in\mathcal{S}$ such that $g(x)\neq 0$.
	Consequently one quickly computes $G(x)=\lim_{\lambda\to\infty} G^\lambda(x)$ for all $x\in\mathcal{S}$ such that $g(x)\neq 0$.
	By dominated convergence we conclude that
	\begin{align*}
	&\inf_{\lambda \in \R} \left( \left(\int_{\mathcal{S}} \|f(x)+\lambda g(x)\|^q \, \mu(dx) \right)^{1/q}-\lambda  \left(\int_{\mathcal{S}} \|g(x)\|^q\,\mu(dx) \right)^{1/q} \right)  \\
	&\le \lim_{\lambda\to \infty} \left(  \int_{\mathcal{S}} \langle f(x)+\lambda g(x), G^{\lambda}(x)\rangle \,\mu(dx)-\lambda  \int_{\mathcal{S}} \langle g(x),G^{\lambda}(x)\rangle \,\mu(dx)\right)\\
	&=\int_{\mathcal{S}} \langle f(x),G(x)\rangle \,\mu(dx)
	\end{align*}
	and the claim follows.
\end{proof}

Let us lastly give the proof of Theorem  \ref{thm:sens} for general seminorms.

\begin{proof}[Proof of Theorem \ref{thm:sens}]
Recall the convention that $\nabla_x \nabla_a f(x,a)\in \R^{k\times d}$ and $\nabla_x f(x,a)\in \R^{d\times 1}$, $\nabla_a f(x,a)\in \R^{k\times 1}$ as well as $h(\cdot)/0=0$.
Further recall that $a^\opt\in\mathcal{A}^\opt(0)$ and $a^\opt_\delta\in\mathcal{A}^\opt(\delta)$ converge to $a^\opt$ as $\delta\to0$.
In order to show
\begin{align*}
\lim_{\delta \to 0}\frac{a^\opt_{\delta}-a^\opt}{\delta}
&=-\Big(\int_{\mathcal{S}} \|\nabla_x f(z,a^\opt)\|^q\,\mu(dz)\Big)^{\frac{1}{q}-1} (\nabla^2_a V(0,a^\opt))^{-1}  \\
&\qquad\cdot \int_{\mathcal{S}} \frac{\nabla_{x}\nabla_a f(x,a^\opt)\, h(\nabla_x f(x,a^\opt))}{\|\nabla_x f(x,a^\opt)\|^{1-q}} \, \mu(dx),
\end{align*}
we first show that for every $i\in\{1,\dots,k\}$
\begin{align}\label{eq:hedgesens}
\lim_{\delta \to 0}\frac{-\nabla_{a_i}V(0,a^\opt_{\delta})}{\delta} 
&= \int_{\mathcal{S}} \nabla_x \nabla_{a_i}f(x,a^\opt)\frac{h(\nabla_x( f(x,a^\opt)) }{ \|\nabla_x f(x),a^\opt)\|^{1-q}}\,\mu(dx)\\
&\qquad\cdot \Big( \int_{\mathcal{S}} \|\nabla_x f(x,a^\opt)\|^q\,\mu(dx)\Big)^{1/q-1},\nonumber
\end{align}
where we recall that $\nabla_{a_i}V(0,a^\opt_{\delta})$ is the $i$-th coordinate of the vector $\nabla_{a}V(0,a^\opt_{\delta})$.
We start with the ``$\le$"-inequality in \eqref{eq:hedgesens}. 
For any $a\in \interior{\mathcal{A}}$, the fundamental theorem of calculus implies that
\begin{align*}
\nabla_{a} f(y,a)-\nabla_{a} f(x,a)&=\int_0^1  \nabla_x \nabla_{a} f(x+t(y-x),a) (y-x) \,dt.
\end{align*}
Moreover, by Lemma \ref{lem:grad_hedge_new} the function $a\mapsto V(\delta,a)$ is (one-sided) directionally differentiable at $a^\opt_{\delta}$ for all $\delta>0$ small and thus for all $i\in \{1, \dots, k\}$
\begin{align}
\label{eq:inequality1}
 & \sup_{\nu \in B_\delta^\opt(\mu,a^\opt_{\delta})} \int_{\mathcal{S}} \nabla_{a_i} f(x,a^\opt_{\delta})\,\nu(dx)\geq 0,
 \end{align}
where we recall $B_\delta^\opt(\mu,a^\opt_\delta)$ is the set of all $\nu\in B_\delta(\mu)$ for which $\int_{\mathcal{S}} f(x,a^\opt_\delta)\,\nu(dx)=V(\delta, a^\opt_\delta)=V(\delta)$.
We now encode the optimality of $\nu$ in $B_\delta^\opt(\mu,a^\opt_\delta)$ via a Lagrange multiplier to obtain
\begin{align}\label{eq:Lagrange}
\begin{split}
&\sup_{\nu \in B_\delta^\opt(\mu,a^\opt_{\delta})} \int_{\mathcal{S}} \nabla_{a_i} f(x,a^\opt_{\delta})\,\nu(dx)\\
&= \sup_{\nu \in B_{\delta}(\mu) }\inf_{\lambda \in \R} \int_{\mathcal{S}}
\big[\nabla_{a_i}f(y,a^\opt_{\delta})+\lambda (f(y,a^\opt_{\delta})
-V(\delta))\big]\nu(dy).
\end{split}
\end{align}
In a similar manner, we trivially have 
\begin{align}\label{eq:Lagrange2}
\int_{\mathcal{S}} \nabla_{a_i}f(x,a^\opt_{\delta})\,\mu(dx)=\int_{\mathcal{S}} \big[\nabla_{a_i}f(x,a^\opt_{\delta}) +\lambda (f(x,a^\opt_{\delta})-V(0,a^\opt_{\delta}))\big]\mu(dx)
\end{align}
for any $\lambda\in\mathbb{R}$, as $\int_{\mathcal{S}} f(x,a^\opt_{\delta}) \,\mu(dx)=V(0,a^\opt_{\delta})$.
Applying \eqref{eq:inequality1} and then \eqref{eq:Lagrange}, \eqref{eq:Lagrange2} we thus conclude for $i\in \{1, \dots, k\}$ 
\begin{align}
&-\nabla_{a_i}V(0,a^\opt_{\delta})
\le\sup_{\nu \in B_{\delta}^\opt(\mu)}\int_{\mathcal{S}} \nabla_{a_i}f(y,a^\opt_{\delta})\,\nu(dy)-\nabla_{a_i}V(0,a^\opt_{\delta})\nonumber\\
&=\sup_{\nu \in B_{\delta}(\mu) }\inf_{\lambda \in \R}\bigg(\int_{\mathcal{S}}
\big[\nabla_{a_i}f(y,a^\opt_{\delta})+\lambda (f(y,a^\opt_{\delta}) -V(\delta))\big] \, \nu(dy)\nonumber\\
&\quad - \int_{\mathcal{S}} \big[\nabla_{a_i}f(x,a^\opt_{\delta}) +\lambda (f(x,a^\opt_{\delta})-V(0,a^\opt_{\delta}))\big]\mu(dx)\bigg)\nonumber\\
\begin{split}
&=
\sup_{\nu \in B_{\delta}(\mu) }\inf_{\lambda \in \R}\bigg(\int_{\mathcal{S}}
\Big[\nabla_{a_i}f(y,a^\opt_{\delta})+\lambda f(y,a^\opt_{\delta})\Big] \,\nu(dy)  \\
&\quad -\int_{\mathcal{S}} \Big[\nabla_{a_i}f(x,a^\opt_{\delta})+\lambda f(x,a^\opt_{\delta})\Big]\,\mu(dx) -\lambda (V(\delta)-V(0,a^\opt_\delta))\bigg). 
\end{split}
\label{eq:boring.2}
\end{align}
As in the proof of Lemma \ref{lem:grad_hedge_new} we note that $B_{\delta}(\mu)$ is compact in $W_{p-\epsilon}^{|\cdot|}$ and both terms inside the $\nu(dy)$ grow at most as $c(1+|y|^{p-\varepsilon})$ by Assumption \ref{ass:sens}.
Thus using \cite[Cor.\ 2, p.\ 411]{min_max_terkelsen1972}  we can interchange the infimum and supremum in the last line above. 
Recall that 
\[ V(\delta)=\sup_{\nu \in B_{\delta}(\mu)}\int_{\mathcal{S}} f(y,a^\opt_\delta)\,\nu(dy),\] 
whence \eqref{eq:boring.2} is equal to
\begin{align*}
&\inf_{\lambda \in \R}\bigg( \sup_{\pi \in C_{\delta}(\mu) } \int_{\mathcal{S}\times\mathcal{S}}
\Big[\nabla_{a_i}f(y,a^\opt_{\delta}) - \nabla_{a_i}f(x,a^\opt_{\delta})+\lambda ( f(y,a^\opt_{\delta})-f(x,a^\opt_{\delta}))\Big]\,\pi(dx,dy)\nonumber \\
&-\lambda \sup_{\pi \in C_{\delta}(\mu) } \int_{\mathcal{S}\times\mathcal{S}} f(y,a^\opt_\delta)-f(x,a^\opt_\delta)\,\pi(dx,dy) \bigg).
\end{align*}
For every fixed $\lambda\in\mathbb{R}$ we can follow the arguments in the proof of Theorem \ref{thm:main} to see that, when divided by $\delta$, the term inside the infimum converges to 
\begin{align}\label{eq:theendisnear}
\left( \int_{\mathcal{S}} \left\| \nabla_x \nabla_{a_i}f(x,a^\opt)+\lambda \nabla_x f(x,a^\opt) \right\|^q\, \mu(dx) \right)^{1/q} -\lambda \left( \int_{\mathcal{S}} \left\| \nabla_x f(x,a^\opt) \right\|^q \, \mu(dx)\right)^{1/q}
\end{align}
as $\delta\to0$.
Note that following these arguments requires the following properties, which are a direct consequence of Assumptions \ref{ass:main} and Assumption \ref{ass:sens}:
\begin{itemize}
  \item $(x,a) \mapsto f(x,a)$ is differentiable on $\interior{\mathcal{S}}\times \interior{\mathcal{A}}$,
  \item $x\mapsto \nabla_{a_i} f(x,a)$ is differentiable on $\interior{\mathcal{S}}$ for every $a\in\cA$,
  \item $(x,a)\mapsto   \nabla_x f(x,a)$ is continuous,
  \item  $(x,a)\mapsto  \nabla_x\nabla_{a_i} f(x,a)$  is continuous,
  \item  for every $r>0$ there is $c>0$ such that $| \lambda \nabla_x f(x,a)|\leq c(1+|x|^{p-1})$ for all $x\in\cS$ and $a\in \cA$ with $|a|\leq r$.
  \item  for every $r>0$ there is $c>0$ such that $| \nabla_x\nabla_{a_i} f(x,a)|\leq c(1+|x|^{p-1})$ for all $x\in\cS$ and $a\in \cA$ with $|a|\leq r$.
  \item For all $\delta\ge 0$ sufficiently small we have $\Aoptim{\delta}\neq\emptyset$ and for every sequence $(\delta_n)_{n\in \N}$ such that $\lim_{n\to \infty} \delta_n=0$ and $(a^\opt_n)_{n\in \N}$ such that $a^\opt_n\in \Aoptim{\delta_n}$ for all $n\in \N$ there is a subsequence which converges to some $a^\opt\in \Aoptim{0}$.
\end{itemize}

Suppose first that $\nabla_x \nabla_{a_i}f(x,a^\opt)=0$ $\mu$-a.s..
Then the right hand side of \eqref{eq:hedgesens} is equal to zero.
Moreover, taking $\lambda=0$ in \eqref{eq:theendisnear}, we also have that $\nabla_{a_i} V(0,a^\opt_\delta)\leq 0$, which proves that indeed the left hand side in \eqref{eq:hedgesens} is smaller than the right hand side. 

Now suppose that $\nabla_x f(x,a^\opt)\neq 0$ $\mu$-a.s..
Then, using the inequality ``$\limsup_\delta \inf_\lambda\leq \inf_\lambda\limsup_\delta$'' and Lemma \ref{lem:inf} to compute the last term (noting that $ \nabla_x f(x,a^\opt)\neq 0$ by assumption), we conclude that indeed the 
\begin{align*}
\limsup_{\delta \to 0} \frac{-\nabla_{a_i}V(0,a^\opt_{\delta}) }{\delta}
&\leq \int_{\mathcal{S}} \nabla_x \nabla_{a_i}f(x,a^\opt)\frac{h(\nabla_x( f(x,a^\opt)) }{ \|\nabla_x f(x),a^\opt)\|^{1-q}}\,\mu(dx)\\
&\qquad\cdot \Big( \int_{\mathcal{S}} \|\nabla_x f(x,a^\opt)\|^q\,\mu(dx)\Big)^{1/q-1}.
\end{align*} 

To obtain the reverse ``$\ge$"-inequality in \eqref{eq:hedgesens} follows by the very same arguments. Indeed, Lemma \ref{lem:grad_hedge_new} implies that 
\begin{align*}
& \inf_{\nu \in B_\delta^\opt(\mu,a^\opt_{\delta})} \int_{\mathcal{S}} \nabla_{a_i} f(x,a^\opt_{\delta})\,\nu(dx)\leq 0
\end{align*}
for all $i \in \{1, \dots, k\}$ and we can write
\begin{align*}
&-\nabla_{a_i}V(0,a^\opt_{\delta})
\ge  \inf_{\nu \in B_\delta^\opt(\mu)} \int_{\mathcal{S}} \nabla_{a_i} f(y, a^\opt_{\delta})\,\nu(dy) -\int_{\mathcal{S}} \nabla_{a_i} f(y,a^\opt_{\delta})\,\mu(dx)\\
&= \inf_{\nu \in B_\delta(\mu)}\sup_{\lambda \in \R} \int_{\mathcal{S}} \big[ \nabla_{a_i} f(y, a^\opt_{\delta})+\lambda(f(y,a^\opt_{\delta})-V(\delta) )\big]\,\nu(dy)  -\int_{\mathcal{S}} \nabla_{a_i} f(x,a^\opt_{\delta})\,\mu(dx)\\
\end{align*}
From here on we argue as in the ``$\le$"-inequality to conclude that \eqref{eq:hedgesens} holds.\\

By assumption the matrix $\nabla_{a}^2 V(0,a^\opt)$ is invertible.
Therefore, in a small neighborhood of $a^\opt$, the mapping $\nabla_a V(0,\cdot)$ is invertible.
In particular 
\[a^\opt_{\delta}=(\nabla_{a} V(0,\cdot))^{-1}\left(\nabla_{a}V(0,a^\opt_{\delta})\right)
\quad\text{and}\quad
a^\opt=(\nabla_{a} V(0,\cdot))^{-1}\left(0\right),\]
where the second equality holds by the first order condition for optimality of $a^\opt$.
Applying the chain rule and using \eqref{eq:hedgesens} gives 
\begin{align*}
\lim_{\delta \to 0}\frac{a^\opt_{\delta}-a^\opt}{\delta}
&=(\nabla_{a}^2V(0,a^\opt))^{-1}\cdot \ \lim_{\delta \to 0}\frac{\nabla_{a}V(0,a^\opt_{\delta})}{\delta}\\
&=- (\nabla^2_aV(0,a^\opt))^{-1}  \Big(\int_{\mathcal{S}} \|\nabla_x f(z,a^\opt)\|^q\,\mu(dz)\Big)^{1/q-1}\\
&\quad  \cdot \int_{\mathcal{S}} \frac{\nabla_{x}\nabla_af(x,a^\opt) h(\nabla_x f(x,a^\opt))}{\|\nabla_x f(x,a^\opt)\|^{1-q}} \, \mu(dx).
\end{align*}
This completes the proof.
\end{proof}

\end{appendix}

\end{document}